\documentclass[11pt,reqno]{amsart}

 \usepackage{amsmath,amsthm,amsfonts}
\usepackage{latexsym,mathabx,txfonts}
\usepackage{amssymb}

\newcommand{\fy}{\varphi}

\newcommand{\p}{\partial}

\newcommand{\R}{\mathbb{R}}

\def\f{\frac}

\numberwithin{equation}{section}

\newtheorem{thm}{Theorem}[section]
\newtheorem{cor}[thm]{Corollary}
\newtheorem{lem}[thm]{Lemma}
\newtheorem{prop}[thm]{Proposition}

\theoremstyle{remark}

\def\less{\lesssim}

\newcommand{\EQ}[1]{\begin{equation}  \begin{split} #1 \end{split} \end{equation} }
\newcommand{\Del}[1]{}

\newcommand{\si}{\sigma}

\def\nn{\nonumber}

\def\eps{\varepsilon}
\def\sign{\mathrm{sign}}

\def\f{\frac}

\def\fy{\varphi}
\def\cong{\equiv}

\begin{document}

\title[Large global solutions]{Large global solutions for energy supercritical nonlinear wave equations on $\R^{3+1}$.}

\author{Joachim Krieger, Wilhelm Schlag}



\thanks{Support of the National Science Foundation DMS-1160817 for the second author, and  the Swiss National Fund for
the first author are gratefully acknowledged. The latter would like to thank the University of Chicago for its hospitality in July 2013, where a portion of this research was conducted.}

\begin{abstract}
For the radial energy-supercritical nonlinear wave equation $$\Box u = -u_{tt} + \triangle u = \pm u^7$$ on $\R^{3+1}$, we prove the existence of a class of global in forward time $C^\infty$-smooth solutions with infinite critical Sobolev norm $\dot{H}^{\frac{7}{6}}(\R^3)\times \dot{H}^{\f16}(\R^3)$.  Moreover, these solutions are stable under suitably small perturbations . We also show that for the {\em defocussing} energy supercritical wave equation, we can construct such solutions which moreover satisfy the size condition $\|u(0, \cdot)\|_{L_x^\infty(|x|\geq 1)}>M$ for arbitrarily prescribed $M>0$. These solutions are stable under suitably small perturbations. 
Our method proceeds by regularization of self-similar solutions which are smooth away from the light-cone but singular on the light-cone. 
The argument crucially depends on the supercritical nature of the equation. 
Our approach should be seen as part of the program initiated in~\cite{KrSchTat1}, \cite{KrSchTat2}, \cite{DoKri}. 
\end{abstract}

\maketitle

\section{Introduction}

We consider in this paper the energy super-critical defocussing/focussing nonlinear wave equation on $\R^{3+1}$, 
\EQ{\label{NLW7}
\Box u\pm u^7=u_{tt}-\Delta u \pm u^7 =0. 
}
The precise power does not play a significant role in the sequel, except for the fact that the problem is energy super-critical. As far as we know, in spite of certain evidence from numerical experiments in the defocussing case that solutions to sufficiently regular but large data appear to stay globally regular, there is no {\it{unconditional}} result asserting global existence of smooth solutions belonging to any class of ``large data", excepting the trivial time periodic solutions in the defocussing case that do not depend on the spatial variable\footnote{These solutions are however most likely unstable under generic perturbations.}. By large data, we mean data which are large in the scaling invariant, hence critical Sobolev space $\dot{H}^{\frac{7}{6}}\times \dot{H}^{\f16}$, and which do not possess some ``hidden" smallness assumption\footnote{Such a smallness assumption can be used to show smallness of suitable critical Strichartz norms for the free wave propagation of the data, which in turn forces the nonlinear solution to essentially behave like a free wave.}, such as  the Besov norm condition on the data $u[0] = (u, u_t)|_{t = 0}$
\EQ{\label{Besov}
\|u[0]\|_{\dot{B}^{\frac{7}{6},2}_{\infty}(\R^3)\times \dot{B}^{\frac{1}{6},2}_{\infty}(\R^3)}<\eps
}
with $\eps$ depending on the size of $\|u[0]\|_{\dot{H}^{\frac{7}{6}}\times \dot{H}^{\f16}}$. 
More precisely, one might consider data $u[0] = (u, u_t)|_{t = 0}$ large provided\footnote{More precisely, the first norm $\|u[0]\|_{\dot{H}^{\frac{7}{6}}(\R^3)\times \dot{H}^{\frac{1}{6}}(\R^3)}$ is assumed to be extremely large compared to the Besov norm $\|u[0]\|_{\dot{B}^{\frac{7}{6},2}_{\infty}(\R^3)\times \dot{B}^{\frac{1}{6},2}_{\infty}(\R^3)}$, and so that the free wave propagation of the data does not have small critical Strichartz norms.} 
\begin{equation}\label{eq:largeCrit}
\|u[0]\|_{\dot{H}^{\frac{7}{6}}(\R^3)\times \dot{H}^{\frac{1}{6}}(\R^3)}\gg 1,\qquad \|u[0]\|_{\dot{B}^{\frac{7}{6},2}_{\infty}(\R^3)\times \dot{B}^{\frac{1}{6},2}_{\infty}(\R^3)}\gtrsim 1
\end{equation}
or also
\begin{equation}\label{eq:infiniteCrit}
\|u[0]\|_{\dot{H}^{\frac{7}{6}}(\R^3)\times \dot{H}^{\frac{1}{6}}(\R^3)} = \infty
\end{equation}
We shall only be interested in $C^\infty$-smooth initial data of precisely this type, although our construction of such data proceeds 
by regularizing certain self-similar solutions which exhibit a singularity on the light-cone. 
Thus if such smooth data satisfy~\eqref{eq:infiniteCrit}, this is due to insufficient decay at infinity, and not to some singular behavior in finite space-time. 
We note here that very sharp global existence results for data satisfying a weak Besov smallness condition such as \eqref{Besov} were derived by F. Planchon in \cite {Plan1}, \cite{Plan2}.  
\\

Our purpose in  this paper is to exhibit a class of $C^\infty$-smooth, global in forward time solutions which obey \eqref{eq:infiniteCrit} and are thus outside the scope of a standard perturbative argument around zero, using the Strichartz framework. Moreover, in the {\it{defocussing case}}, we show that these solutions can be forced to have arbitrarily large amplitude\footnote{Observe that any nonzero solution can be forced to have large amplitude near the origin by re-scaling it. However, large amplitude far away from the origin corresponds (in the radial case) in some sense to ``large solutions".} on the set $\{|x|\geq 1\}$. Our argument for the first result hinges crucially on the energy super-critical nature of the equation, and for the second uses both the defocussing as well as the supercritical character. As a byproduct of our method we also obtain the stability of our solutions with respect to suitably mild perturbations. 
The main results of this paper are the following theorems.  

\begin{thm}\label{thm:Main1}
For both the defocussing/focussing supercritical nonlinear wave equation \eqref{NLW7} on $\R^{3+1}$, there exist smooth data sets $(f, g)\in C^\infty\times C^\infty$ decaying at infinity to zero and satisfying $$\|(f, g)\|_{\dot{H}^{\frac{7}{6}}(\R^3)\times \dot{H}^{\frac{1}{6}}(\R^3)} = \infty \text{\ \  but\ \ } 
\|(f, g)\|_{\dot{H}^{s}(\R^3)\times \dot{H}^{s-1}(\R^3)}<\infty$$ for any $s>\frac{7}{6}$, and such that the corresponding evolution of \eqref{NLW7} exists globally in forward time as a $C^\infty$-smooth solution. These solutions are stable under a certain class of perturbations. 
\end{thm}

We note that the solutions established by this theorem satisfy 
\[
\|f\|_{L_x^\infty(|x|\geq 1)}\ll 1
\]
In the following theorem we find solutions which are ``more nonlinear", as evidenced by a highly oscillatory character. 
This remark will become clearer in Section~\ref{sec:12}. 

\begin{thm}\label{thm:Main2}
Let $M>0$ be given arbitrarily. 
For the defocussing supercritical nonlinear wave equation \eqref{NLW7} on $\R^{3+1}$, there exist smooth data sets $(f, g)\in C^\infty\times C^\infty$ decaying at infinity to zero and satisfying $$\|(f, g)\|_{\dot{H}^{\frac{7}{6}}(\R^3)\times \dot{H}^{\frac{1}{6}}(\R^3)} = \infty \text{\ \ but\ \ }
\|(f, g)\|_{\dot{H}^{s}(\R^3)\times \dot{H}^{s-1}(\R^3)}<\infty$$ for all $s>\frac{7}{6}$, as well as 
\EQ{\label{fM}
\|f\|_{L_x^{\infty}(|x|\geq 1)}>M
}
and such that the corresponding evolution of \eqref{NLW7} exists globally in forward time as a $C^\infty$-smooth solution. These solutions are stable under a certain class of perturbations and they are not small in the Besov sense~\eqref{Besov}. 
\end{thm}

Let us also formulate two of the statements that follow from the methods of this paper for the context of smooth compactly supported data.

\begin{thm}
\label{thm:intro1}
Consider the defocusing equation~\eqref{NLW7}. 
For any $M>0$ there exist smooth compactly supported radial data $(f,g)$ with support in $B_2(0)$ and  
\EQ{\label{gross}
\| (f,g) \|_{\dot{H}^{\frac{7}{6}}\times \dot{H}^{\frac{1}{6}}(\R^3)} >M,
}
so that  \eqref{NLW7} admits a smooth solution $u$ for all times $0\le t\le 1$, which furthermore satisfies 
\EQ{\label{eq:StIn}
\inf_{\frac{1}{p}+\frac{1}{q}=\frac{1}{2},\,p>2}\|\,|\nabla_x|^{\alpha(q)}u\|_{L_t^pL_x^q([0,1]\times\R^3)}\geq 1
}
where $\alpha(q) =  \frac{2}{q} + \frac{1}{6}$. Moreover, the data can be chosen from an
open nonempty set relative to the norm in~\eqref{gross}. 
\end{thm}

The space-time norms in \eqref{eq:StIn} are examples of Strichartz norms relevant in this context.  In fact, we may include any other
admissible Strichartz norms  in the infimum in~\eqref{eq:StIn}, as well as in the following theorem. 

\begin{thm}
\label{thm:intro2}
Consider the defocussing equation \eqref{NLW7} with the $+$-sign. 
For any $M_1, M_2>0$ there exist smooth compactly supported radial data $(f,g)$ with support in some $B_K(0)$, $K\geq 1$ such that the we have 
$$\|(f, g)\|_{\dot{H}^{\frac{7}{6}}\times \dot{H}^{\frac{1}{6}}(\R^3)}>M_1,$$ the evolution of these data exists on $0\leq t\leq \frac{K}{2}$ as a smooth function, and moreover  we have, with $\alpha(q)$ as in the previous theorem, 
\[
\inf_{\frac{1}{p}+\frac{1}{q}=\frac{1}{2},\,p>2}\| \, |\nabla_x|^{\alpha(q)}u\|_{L_t^pL_x^q([0,K/2]\times\R^3)}\geq 1,
\] 
as well as 
\[
\|f\|_{L_x^\infty(|x|\geq 1)}> M_2
\]
\end{thm}

Note that the inequality 
\[
\inf_{\frac{1}{p}+\frac{1}{q}=\frac{1}{2},\,p>2}\|\,|\nabla_x|^{\alpha(p, q)}u\|_{L_t^pL_x^q([0,K/2]\times\R^3)}\geq 1
\]
means that all the scale invariant Strichartz norms of the solution are not small, precluding a simple perturbative argument around the free wave propagation of the initial data. Furthermore, the condition on the support precludes a simple construction piecing together small solutions in disjoint light cones. 
In fact, the philosophy of this work is to use a perturbative approach around suitably constructed {\it{elliptic nonlinear objects}}, in this case approximate self-similar solutions. More precisely, as for the method employed in \cite{DoKri}, the idea is to use {\it{special singular solutions}}, which are obtained by making a self-similar ansatz, to generate non-trivial global dynamics via a carefully chosen regularization and solution of a perturbative problem. In fact, the regularization destroys the scaling invariance and this turns out to be important for the ensuing perturbative argument. We observe also that the method of \cite{DoKri} grew directly out of the methods introduced in \cite{KrSchTat1}, \cite{KrSchTat2}. 
In our present context, however, we do not rely on the  spectral methods and parametrix constructions used in these references, but rather rely on the standard Strichartz and energy estimates.

In the following section we construct smooth self-similar solutions of the form 
\EQ{\label{globalQ}
u_0(t,r)=t^{-\f13} Q\big(\f{r}{t}\big), \text{\ \ either \ \ } r<t \text{\ \ or\ \ }r>t
}
by a reduction to a  nonlinear Sturm-Liouville problem, see~\eqref{ODE7}.  We solve this ODE by
contraction off of the leading linear behavior assuming smallness in~$L^\infty$. This smallness also
allows us to solve a nonlinear connection problem at an intermediate point such as $a:=\f{r}{t}=\f12$ by 
the inverse function theorem. 

As we shall see, starting with small data at $a = 0$, $Q(a)$ exhibits a singularity of the form $|1-a|^{\f23}$ near $a=1$ which
precisely fails logarithmically to belong to the scaling critical Sobolev space $\dot{H}^{\frac{7}{6}}(\R^3)$,
and its time-derivative fails logarithmically to belong to $\dot{H}^{\frac{1}{6}}(\R^3)$. This part of the construction
does not depend on super-criticality in any way. In fact, it can be carried out in other dimensions and for
other powers. In each case, the singularity will fall logarithmically outside of the scaling critical space. 
For example, in $\R^5$ for the $\dot H^2\times \dot H^1$-critical $u^5$ equation the singularity is of the form $|1-a|^{\f32}$,
whereas for the same equation in~$\R^3$ (the energy critical one), the singularity is $|1-a|^{\f12}$. 

In the second part of the construction we first glue together the two solutions residing inside and outside the light-cone,
respectively, at $r=t$ to form a continuous function $u_0(t,r)$, which decays as $r\to\infty$ at the rate~$r^{-\f13}$ (and thus
fails to lie in $\dot H^{\f76}$ at $r=\infty$); the decay~$r^{-\f13}$ is the generic one, we may also achieve~$r^{-\f43}$, but
then the time-derivative fails to belong to $\dot H^{\f16}$ at $r=\infty$. 

 We then multiply the singular components of $u_0(t,r)$  by a smooth cutoff
equal to~$1$ away from $|r-t|\le 2C$ and vanishing on $|r-t|\le C$, say. This smooth function $u_1(t,r)$ no longer solves~\eqref{NLW7},
but we show that we may add a smooth correction $v(t,r)$ to $u_1(t,r)$ so that $$u(t,r)=u_1(t,r)+v(t,r)$$ does solve~\eqref{NLW7}. This
part of the argument does crucially depend on the energy supercritical nature of the problem (although neither the exact power nor the focussing/defocussing character is relevant). 
This perturbative argument relies on an interplay between the scaling critical norm and the standard energy. We remark that
the latter restricted to $r<t$ grows like $t^{\f13}$ as $t\to\infty$ due to incoming waves.

In a final part of the paper, we re-consider the self-similar solutions on the outside of the light cone, but only in the {\it{defocussing case}}. We show that one of the parameters determining the solution near the singularity at $a=1$ can be chosen arbitrarily large, leading to rapid growth and oscillation of the solution on the set $a>1$ but near $a$. The defocussing character of the problem permits to extend these solutions all the way to $a \rightarrow +\infty$, where they again decay asymptotically like $a^{-\frac{1}{3}}$. 
We show that such a ``large self-similar solution" can be glued to a ``small self-similar solution" inside the light cone. Truncating (parts of) this continuous function to make it $C^\infty$-smooth just as before, we then show that we can construct an exact $C^\infty$ solution with just the behavior detailed in Theorem~\ref{thm:Main2}. The key to obtaining the smallness gain for the nonlinear estimates comes from choosing the time $t\geq T$ large enough. 

\smallskip

We cannot possibly do justice to the large body of work that has been devoted to studying the equation
\[
\Box u\pm |u|^{p-1}u=u_{tt}-\Delta u \pm |u|^{p-1}u =0
\]
in $\R^{3+1}$ (or other dimensions) for smooth, compactly supported data over the past 50 years. In the defocussing case, J\"orgens~\cite{Jo} showed
global existence for $p<5$, the subcritical regime. 
Struwe~\cite{Struwe} then settled the energy critical case $p=5$ radially, and Grillakis~\cite{Grill} nonradially. See the book by Shatah, Struwe~\cite{ShSt} for
an account of these developments.  
A very general method to attack energy critical problems and in particular recover the result of Struwe and Grillakis was developed recently by Kenig and Merle in \cite{KeMe0}. 
A much more quantitative approach, implying scattering and global space-time bounds explicitly in terms of the energy, but more contingent on the specific structure of the equation, was established in the work \cite{Bou} by Bourgain in the context of the energy-critical defocusing radial nonlinear Schrodinger equation. These methods were then further developed by T. Tao in \cite{Tao} to treat a ``slightly super-critical wave equation'' (where the critical nonlinearity is multiplied by a logarithmic factor). 
In this context, we also mention Struwe's recent work on energy super-critical wave equations on $\R^{2+1}$ with exponential type nonlinearities, \cite{Struwe1}, \cite{Struwe2}. Observe that all pure power nonlinear wave equations on $\R^{2+1}$ of the form $\Box u = \pm |u|^{p-1}u$, $p>1$, are energy-subcritical.\\  
Lebeau~\cite{Leb1,Leb2}  studies instability of solutions to semi-linear equations including the supercritical equations such as~\eqref{NLW7}, again in the defocussing case, relative
to weaker norms than the scaling critical ones. We remark that the self-similar solutions constructed in the following section
belong to all spaces of the form 
\[
\dot{H}^{\frac{7}{6}-\eps}(\R^3)\times \dot{H}^{\frac{1}{6}-\eps}(\R^3)
\]
with $\eps>0$ provided we restrict them to the interior of the light-cone. It is conceivable that this might allow one to obtain
aspects of the supercritical ill-posedness results as in Lebeau's work by solving backward from $t=1$ to $t=0$ inside of the cone. However, we do not pursue such matters here. 

By Strichartz theory, cf.~Lemma~\ref{lem:ST}, the equation~\eqref{NLW7} is globally well-posed for smooth compactly supported data with small critical norm (in both
the focusing and defocusing cases).  It is also locally well-posed for any data in that norm, and the solutions preserve regularity and obey the finite propagation speed. 
Kenig, Merle~\cite{KM} proved  for \eqref{NLW7} and the defocussing case that break-down of smooth solutions in finite time $T$ 
can only occur provided 
 \begin{equation}\nn
 \sup_{0<t<T} \|(u(t), u_t(t))\|_{\dot{H}^{\frac{7}{6}}(\R^3)\times \dot{H}^{\frac{1}{6}}(\R^3)} = \infty
\end{equation}
This work has generated many further developments of a similar character, see for example the recent work \cite{Bu1}. 
Bizo\'n, Maison, and Wasserman~\cite{B} establish an infinite family of smooth solutions
for the focusing supercritical equation~\eqref{NLW7} which are obtained by rescaling of a fixed
profile. In essence, these authors observe via an ODE analysis that next the to ODE blowup $c(T-t)^{-\f13}$
present in the focusing equation~\eqref{NLW7}, this equation also allows for infinitely many solutions obtained
from this one by multiplication with a time-dependent non-constant profile of the form $U(r/(T-t))$. It is shown
in \cite{B} that there exists an infinite sequence of values $U(0)$ and~$U(1)$ which give rise to a smooth solution
of~\eqref{NLW7}.   
We also mention here the works by Donninger and Sch\"orkhuber for the focusing supercritical wave equation, where they establish stability of the explicit ODE blow up solutions \cite{DoSch}. 

However, the investigations of this paper go  in a very different direction
since we are mainly concerned with the defocusing equation and global smooth solutions, as opposed to finite time blow up.

\section{Self-similar solutions}\label{sec:selfsim}

\subsection{The interior light-cone}
  We seek a solution of \eqref{NLW7} of the form $u_0(t,r)=t^{-\frac13} Q(r/t)$ for $0\le r<t$. In general, we expect these solutions to be singular at least on the light-cone, i.e., at  $\frac{r}{t} = a = 1$, and a precise description of this failure of regularity shall play a key role later on. 
To begin with, $Q$ satisfies the ODE on $0\le a<1$
\EQ{\label{ODE7}
(a^2-1)Q''(a)+\Big(\frac83a - \frac{2}{a}\Big) Q'(a)+\f49 Q(a) \pm  Q(a)^7=0
}
The natural initial conditions at $a=0$ are
\EQ{\label{init}
Q(0)=q_0>0, \quad Q'(0)=0
}
We shall first solve this initial value problem on the interval $0\le a\le 1/2$ which leads to a $1$-parameter
family of solutions.  We then solve
the nonlinear connection problem at $a=1/2$ with a $2$-parameter family of solutions on the interval $(1/2,1)$.
The two parameters are important, since they allows us to apply the inverse function theorem.  

\begin{lem}\label{lem:smalla}
There exists $\eps>0$ small such that  for any $0\le q_0\le \eps$ the equation~\eqref{ODE7} admits a unique smooth solution on $[0,1/2]$ 
with initial conditions~\eqref{init}. 
Moreover, 
\EQ{ \label{Q1/2}
Q(1/2) &=q_0Q_0(1/2)+O(q_0^7) \\
Q'(1/2) &=q_0Q_0'(1/2)+O(q_0^7)
}
and the solution extends as a smooth even function the the interval $[-1/2,1/2]$. 
\end{lem}
\begin{proof}
The associated homogeneous linear equation is 
\EQ{\label{ODE}
(a^2-1)Q''(a)+\Big(\frac83a - \frac{2}{a}\Big) Q'(a)+\f49 Q(a) =0
}
with fundamental system
\EQ{\label{fy}
\fy_1(a)=a^{-1}(1-a)^{\f23},\quad \fy_2(a)=a^{-1}(1+a)^{\f23}
}
Define the Green function for $0<b<a<1$: 
\EQ{\label{Green}
G(a,b) &:= \frac{\fy_1(a)\fy_2(b)- \fy_1(b)\fy_2(a)}{W(b) (b^2-1)} \\
W(a) &:=  \fy_1(a)\fy_2'(a) - \fy_1'(a)\fy_2(a)
}
It has the property that the inhomogeneous equation
\EQ{\label{ODEf}
&(a^2-1)\fy''(a)+\Big(\frac83a - \frac{2}{a}\Big) \fy'(a)+\f49 \fy(a) =f(a)\\
& \fy(0)=0, \; \fy'(0)=0
}
is solved by 
\[
\fy(a) = -\int_0^a G(a,b)f(b)\, db
\]
We therefore seek a solution of \eqref{ODE7} on $0\le a\le \f12$ with initial conditions~\eqref{init} of the form
\EQ{\label{Volt}
Q(a) &= \f34 q_0(\fy_1(a)-\fy_2(a)) \pm \int_0^a G(a,b) Q(b)^7\, db
}
Note that $Q_0(a):=\f34(\fy_2(a)-\fy_1(a))$ is analytic and even around $a=0$.  Moreover, $Q_0(0)=1$. 
Assume $0\le q_0\le \eps$ and define the space
\[
X_{q_0}:=q_0 Q_0 + \{  h(a)\mid h\in C^2([0,1/2]), \; \|h\|_{C^2}\le q_0^6, \; |h(a)|\le q_0^6 a^2\}
\]
We equip the linear  space defined by the set on the right-hand side with the norm 
\[
\|h\|_{C^2} + \sup_{0<a<\f12} a^{-2}|h(a)|
\]
Our main claim is as follows: {\em there exists $\eps>0$ small such that  for any $0\le q_0\le \eps$ the equation \eqref{Volt} has a
unique solution in $X_{q_0}$. }

By explicit calculation,
\EQ{\label{W}
W(a) = \f{4}{3a^2}(1-a^2)^{-\f13}
}
and $(W(b)(b^2-1))^{-1}$ is analytic on $(-1, 1)$ with expansion
\[
-\f34b^2-\f12b^4+O(b^6)
\]
as
 $b\to0$. Second, for $0<b<a$, 
 \[
 G(a,b) = \f{b}{a} (-\f34 + O(b^2)) \big[ (1-a)^{\f23} (1+b)^{\f23} -  (1-b)^{\f23} (1+a)^{\f23}   \big]
 \]
whence in particular 
$|G(a,b)|\le Cb$ for all $0<a\le\f12$.  Moreover, setting $b=ua$ with $0<u<1$ shows that
\[
\tilde G(a,u) := G(a,au) 
\]
is smooth in $|a|<1$ and $|u|<1$ and satisfies the bound
\[
\max_{|u|\le 1} \tilde G(a,u) \le C|a|. 
\]
Therefore,  
\EQ{\label{Gab}
\int_0^a |G(a,b)|\, db = a \int_0^1 |\tilde G(a,u)|\, du \le Ca^2
}
Define
\[
(Tf)(a):=q_0 Q_0(a) \pm \int_0^a G(a,b) f(b)^7\, db= q_0 Q_0(a) \pm a\int_0^1 \tilde G(a,u) f(au)^7\, du
\]
We claim that $T$ is a contraction in $X_{q_0}$ and therefore  has a fixed point $f\in X_{q_0}$.
Any $f\in X_{q_0}$ satisfies $|f(a)|\le Mq_0$ for all $0\le a\le \f12$ where $M$ is some absolute constant. Thus, 
\[
h(a):= \int_0^a G(a,b) f(b)^7\, db
\]
satisfies by \eqref{Gab}
\[
|h(a)|\le CM^7a^2 q_0^7 \ll q_0^6\, a^2, \; |h'(a)|\le CM^7a q_0^7 \ll q_0^6\, a
\]
as well as 
\[
|h''(a)|\le CM^7  q_0^7 \ll q_0^6
\]
provided $q_0$ is small. Hence, $T:X_{q_0}\to X_{q_0}$.  For the contraction, we estimate
\EQ{\nn
\| Tf-Tg\|_{X_{q_0}} &\le C(\|f\|_\infty + \|g\|_\infty)^6\|f-g\|_\infty \\
&\le CM^6 q_0^6 \|f-g\|_{X_{q_0}}
}
For $q_0$ small this implies that $T$ is a contraction and we are done with our main claim.
We note from the integral equation that $f$ is even on $[-1/2,1/2]$. 

As for the higher regularity, this of course follows form standard regularity results. We proceed by 
induction in the number of derivatives. Starting from the integral equation
\[
f(a)=q_0 Q_0(a) \pm a\int_0^1 \tilde G(a,u) f(au)^7\, du
\]
we observe that 
\EQ{\label{eq:k}
f^{(k)}(a)= q_0 Q_0^{(k)}(a) +H_k(a)\pm 7a\int_0^1 \tilde G(a,u) f(au)^6 f^{(k)}(au)u^k\, du
}
for any integer $k\ge0$ where $H_k$ is smooth; one has $H_0=0$ and 
\[
\pm H_1(a)= \int_0^1 \tilde G(a,u) f(au)^7\, du + a\int_0^1 \tilde G_a(a,u) f(au)^7\, du
\]
and so forth. Clearly, $H_k$ only involves $k-1$ derivatives of~$f$ and is therefore small in the norm of
continuous function on the interval $[0,1/2]$ by the inductive assumption. We can therefore contract~\eqref{eq:k}
to produce a continuous small solution~$f^{(k)}(a)$ on $[0,1/2]$. This shows that $f$ possesses any number of
derivatives. 
\end{proof}

The solution is in fact analytic. We remark that one can also solve~\eqref{ODE7} near $a=0$ (and thus
also on $[0,1/2)$) by power series. Writing the usual iteration for the coefficients shows that they are all positive. This
is a reflection of the defocusing nature of~\eqref{NLW7}. 
Thus, the solution  is monotone increasing together with all derivatives.  We have chosen to use the Green function since the nonlinear
recursion is not entirely elementary. 
Next, we solve backwards starting from $a=1$. 

\begin{lem}\label{lem:near1}
Given $q_1,q_2\in (-\eps,\eps)$ there exists a unique solution $Q(a)$ of \eqref{ODE7}  on $[\f12,1)$ of the form
\EQ{ \label{anear1}
Q(a) & =  (1-a)^{\f23} Q_1(a) +  Q_2(a) + (1-a)^{\f73} Q_3(a)
}
with $Q_1, Q_2, Q_3\in C^\infty([\f12,1])$ and 
\EQ{\nn
Q_1(a) &= q_1(1+O(1-a)),\;\; Q_2(a)= q_2(1+O(1-a)),\\
Q_{3}(a) &=(|q_1|^7+|q_2|^7)O(1)
}
where the $O(\cdot)$ terms are smooth functions in~$a\in [1/2,1]$. 
Finally,
\EQ{\label{IMP}
Q(1/2) &= q_1 \fy_1(1/2) + q_2 2^{-\f23} \fy_2(1/2) + O(|q_1|^7+|q_2|^7)\\
Q'(1/2) &= q_1 \fy_1'(1/2) + q_2 2^{-\f23} \fy_2'(1/2) + O(|q_1|^7+|q_2|^7)
}
where $\fy_1,\fy_2$ are the functions from~\eqref{fy}. 
\end{lem}
\begin{proof}
We convert the ODE to the following integral equation
\EQ{\label{Volt1}
Q(a) &=  q_1 \fy_1(a) + q_2 2^{-\f23} \fy_2(a) \mp \int_a^1 G(a,b) Q(b)^7\, db
}
where $G$ is the Green function from~\eqref{Green}. Since in this case $a<b<1$, the integral comes with a  negative sign.  

By inspection, $(1-a)^{-\f23}\fy_1(a)=1/a$, $2^{-\f23} \fy_2(a)=((1+a)/2)^{\f23}/a$ are analytic on $a>0$, and equal~$1$ at $a=1$.  
Furthermore, taking the Wronskian~\eqref{W} into account, the Green function \eqref{Green} is of the form
\EQ{\label{Gab2}
G(a,b)= g_1(a,b) + (1-b)^{\f23} (1-a)^{-\f23} g_2(a,b)
}
where $g_1, g_2$ are smooth on $a,b\in [1/2,1]$.  If $\omega(b)$ is smooth on $[1/2,1]$, then
\EQ{\label{int0}
\int_a^1 G(a,b)\omega(b)\, db = O(1-a)
}
is smooth on $[1/2,1]$, as well as 
\EQ{\label{int12}
\int_a^1 G(a,b)\omega(b)(1-b)^{\f{2k}{3}} \, db = (1-a)^{\f{2k}{3}} O(1-a),\quad k=1,2
}
where $O(1-a)$ is a smooth function of $a\in [1/2,1]$. 
This allows one to convert \eqref{Volt1} into a system for $Q_1, Q_2,Q_3$ which
we again solve by contraction.  To be specific,
\EQ{\nn
[(1-a)^{\f23} Q_1(a) +  Q_2(a)+ (1-a)^{\f73} Q_3(a)]^7 &= \sum_{j=0}^2 (1-a)^{\f{2j}{3}} N_j(a,Q_1, Q_2,Q_3)
}
Here each $N_j(a,Q_1, Q_2,Q_3)$ is a linear combination  of terms of the form
\[
(1-a)^m Q_1^{k_1}(a) Q_2^{k_2}(a) Q_3^{k_3}(a)
\]
where $k_1+k_2+k_3=7$. In particular, if all $Q_i$ are smooth, then $N_j$ is, too.   For example, $N_2$ contains the term $21Q^2_1(a) Q^5_2(a)$.  We remark that in each $N_{j}$ the function $Q_{3}$ appears with a factor of at least $(1-a)$.
For example, the term
\EQ{\label{Q3 besser}
7  Q_2^{6}(a)(1-a)^{\f73} Q_3(a) = (1-a)^{{\f43}} 7 (1-a) Q_2^{6}(a) Q_3(a)  
}
contributes $7 (1-a) Q_2^{6}(a) Q_3(a) $ to $N_{2}$. 
We now solve for $Q_j$ in the following form
\EQ{\label{drei sys}
Q_1(a) &=   q_1 a^{-1}  \mp (1-a)^{-\f23}  \int_a^1 G(a,b) (1-b)^{\f23} N_1(b,Q_1, Q_2,Q_3) \, db \\
Q_2(a) &= q_2 2^{-\f23} \fy_2(a) \mp \int_a^1 G(a,b) N_0(b,Q_1, Q_2,Q_3)\, db \\
Q_3(a) &= \mp (1-a)^{-\f73} \int_a^1 G(a,b)  (1-b)^{\f43} N_2(b,Q_1, Q_2,Q_3)\, db
}
By \eqref{int0}, \eqref{int12}  the right-hand sides are smooth if the $Q_{j}$ are.  We write the system~\eqref{drei sys}
in the fixed-point form $\vec Q= T(\vec Q)$ where $T$ denotes the column vector of the right-hand sides and $\vec Q:=(Q_{1},
Q_{2}, Q_{3})$. 

We set up a contraction for $T$ in the space of continuous functions on the interval $[1/2,1]$.  For $\eps>0$ we
find a unique solution  of the form 
\EQ{\nn 
Q_{1}(a) &= q_{1}a^{-1} + (|q_{1}|+|q_{2}|)^{7} (1-a) R_{1}(a) \\
Q_{2}(a) &= q_2 2^{-\f23} \fy_2(a) + (|q_{1}|+|q_{2}|)^{7} (1-a) R_{2}(a) \\
Q_{3}(a) &=   (|q_{1}|+|q_{2}|)^{7} R_{3}(a)
}
where $R_{j}$ are continuous and satisfy $|R_{j}(a)|\le M$  on $[1/2,1]$, where $M$ is some absolute constant.

Inserting these representations into \eqref{drei sys} implies  that we  gain at least one degree of regularity at $a=1$, in 
other words, one factor of~$(1-a)$. For the terms involving $R_1, R_2$ this is clear, since each application of the integration in~\eqref{drei sys}
gains a factor of~$(1-a)$.  On the other hand, for the $Q_{3}$ term we need to use the observation~\eqref{Q3 besser}, i.e., 
the fact that $Q_{3}$ carries at least a factor of $(1-a)$ when
reinserted into the nonlinearity $N_{2}$. Repeating this procedure produces more and more smoothness at $a=1$. 
The smoothness for $1/2\le a<1$ is clear.   
\end{proof}

We can now solve \eqref{ODE7} and thus obtain the special self-similar solutions of~\eqref{NLW7}.  
The following corollary shows that such solutions (nonzero of course), necessarily exhibit the $(1-a)^{\f23}$ singularity at~$a=1$. 

\begin{cor}
\label{cor:1}
For any small $q_0$ the ODE \eqref{ODE7} has a unique $C^2$ solution $Q(a)$ on $[0,1)$ with $Q(0)=q_0$ and $Q'(0)=0$.
This solution is of the form~\eqref{anear1} near $a=1$.  We can have neither $q_1=0$ nor $q_2=0$. 
\end{cor}
\begin{proof}
To prove this, let $Q$ be the solution for given small $q_0$ as generated by Lemma~\ref{lem:smalla}.
By the inverse function theorem, we may find $q_1,q_2$ small so that \eqref{IMP} matches the values given by~\eqref{Q1/2}.
The application of the inverse function theorem is justified since the derivative in $q_1, q_2$ at $(q_1,q_2)=0$ of~\eqref{IMP} is the Wronskian of $\fy_1,\fy_2$,
which does not vanish.  The final claim is seen for the same reason: we cannot achieve linear dependence of the
solutions generated by Lemmas~\ref{lem:smalla} and~\ref{lem:near1} when either $q_1=0$ or $q_2=0$. 
\end{proof}

In particular, these solutions logarithmically  fail to belong to $\dot H^{\f76}(\R^3)$. 

\subsection{The exterior light-cone}
We next carry out a similar construction in the region $r>t$.   Here $a=r/t>1$, but the analysis is essentially the same. We
begin of the analogue of Lemma~\ref{lem:near1}.  

\begin{lem}\label{lem:near2}
Given $\tilde q_1, \tilde q_2\in (-\eps,\eps)$ there exists a unique solution $Q(a)$ of \eqref{ODE7}  on $(1,2]$ of the form
\EQ{ \label{anear2}
Q(a) & =  (a-1)^{\f23} \tilde Q_1(a) +  \tilde Q_2(a) + (a-1)^{\f73} \tilde Q_3(a)
}
with $\tilde Q_1, \tilde Q_2, \tilde Q_3\in C^\infty((1,2])$ and 
\EQ{\nn
\tilde Q_1(a) &= \tilde q_1(1+O(a-1)),\;\; \tilde Q_2(a)= \tilde q_2(1+O(a-1)),\\
\tilde Q_{3}(a) &=(|\tilde q_1|^7+|\tilde q_2|^7)O(1)
}
where the $O(\cdot)$ terms are smooth functions in~$a\in [1,2]$. 
Finally,
\EQ{\label{IMP2}
Q(2) &= \tilde q_1 \tilde \fy_1(2) + \tilde q_2 2^{-\f23} \fy_2(2) + O(|\tilde q_1|^7+|\tilde q_2|^7)\\
Q'(2) &= \tilde q_1 \tilde \fy_1'(2) + \tilde q_2 2^{-\f23} \fy_2'(2) + O(|\tilde q_1|^7+|\tilde q_2|^7)
}
where $\tilde \fy_1(a)=a^{-1}(a-1)^{\f23}$ and $\fy_2$ is as in~\eqref{fy}. 
\end{lem}
\begin{proof}
The proof is analogous to that of Lemma~\ref{lem:near1} and we skip it. 
\end{proof}

Next, we glue this solution together with one on $2\le a<\infty$. 

\begin{lem}\label{lem:largea}
There exists $\eps>0$ small such that  for any $|m_1|, |m_2|\le \eps$ the equation~\eqref{ODE7} admits a unique smooth solution on $[2,\infty)$ 
so that  as $a\to\infty$
\EQ{\label{a infty}
Q(a)= m_1 \tilde \fy_1(a) + m_2 \fy_2(a) + O(a^{-\f73})
}
and 
\EQ{ \label{Q2}
Q(2) &=m_1 \tilde \fy_1(2)  +  m_2  \fy_2(2) +O((|m_1|+|m_2|)^7) \\
Q'(2) &=m_1 \tilde \fy_1'(2)  +   m_2  \fy_2'(2) +O((|m_1|+|m_2|)^7)
}
Here $\tilde\fy_1,\fy_2$ are as in Lemma~\ref{lem:near2}.
\end{lem}
\begin{proof}
We use the Green function~\eqref{Green} but defined in terms of $\tilde \fy_1, \fy_2$: 
\EQ{\label{Green2}
G(a,b) &:= \frac{\tilde\fy_1(a)\fy_2(b)- \tilde\fy_1(b)\fy_2(a)}{W(b) (b^2-1)} \\
W(a) &:=  \tilde \fy_1(a)\fy_2'(a) - \tilde \fy_1'(a)\fy_2(a)
}
The denominator is $W(b)(b^2-1)$, which decays at the rate~$b^{-\f23}$ as $b\to\infty$. 
The perturbative ansatz is 
\EQ{\label{m1m2}
Q(a) = m_1 \tilde \fy_1(a) + m_2 \fy_2(a) \pm \int_a^\infty G(a,b) Q(b)^7\, db
}
This is solved by contraction, and the asymptotics~\eqref{a infty} follows by 
inserting the two types of asymptotic behaviors exhibited by~$G(a,b)$, i.e., $a^{-\f43}b^{\f13}$, and $a^{-\f13}b^{-\f23}$. 
Integrating these against $Q(b)^7$ which decays at least as fast as $b^{-\f73}$ then shows that the
integral in~\eqref{m1m2} decays like $a^{-\f73}$.  Moreover, we obtain~\eqref{Q2} by setting~$a=2$.
\end{proof}

Finally, we glue the two solutions together to obtain one on the whole interval $a>1$.  The following corollary is
an immediate application of Lemmas~\ref{lem:largea} and~\ref{lem:near2}. 

\begin{cor}\label{cor:2}
For any small $m_1,m_2$ there exists a smooth solution $Q(a)$ to the ODE~\eqref{ODE7} on $1<a<\infty$, with the asymptotics~\eqref{a infty}
as $a\to\infty$.  As $a\to1$ the solution obeys the representation~\eqref{anear2}. The map $(m_1,m_2)\mapsto (\tilde q_1,\tilde q_2)$
is a diffeomorphism from a small neighborhood of $(0,0)$ to another.  Finally, there exists a linear map $m\mapsto (m_1,m_2)$
so that for every small~$m$ the corresponding solution decays like~$m\,a^{-\f43}$ as $a\to\infty$. 
\end{cor}
\begin{proof}
As for the interior light-cone, we solve the connection problem at $a=2$ by means of the inverse function theorem. 
This is legitimate again by smallness as well as the non-vanishing of the Wronskian.  In general, we obtain a $2$-parameter family. 
But we may cancel the leading order $a^{-\f13}$ as $a\to\infty$ by means of a linear relation between $m_1, m_2$. This is
the claim relating to a linear map $m\mapsto (m_1,m_2)$, and produces decay at the rate $a^{-\f43}$. The result is a $1$-parameter family
of solutions. 
\end{proof}

\subsection{Matching at the light-cone}

Combining Corollaries~\ref{cor:1}, \ref{cor:2} leads to the following conclusion.  
For the meaning of the parameters $q_1, q_2, m_1, $ etc.\ see these corollaries. 

\begin{cor}\label{cor:3}
For any small $q_0$ the ODE \eqref{ODE7} has a unique $C^2$ solution $Q(a)$ on $[0,1)$ with $Q(0)=q_0$ and $Q'(0)=0$.
There exist  infinitely many continuous extensions of $Q(a)$ to $a\ge1$ which solve \eqref{ODE7} on $a>1$ and decay 
at least at the rate $a^{-\f13}$ as $a\to\infty$. These extensions are given by Corollary~\ref{cor:2}. 
The global solutions on $a\ge0$ satisfy $q_2=\tilde q_2$, in the notation of Lemma~\ref{lem:near1}. 
We denote these functions on $a\ge0$ by $Q_0(a)$ and we have the global representation 
\EQ{\label{Q0}
Q_0(a) = |1-a|^{\frac{2}{3}}\big[Q_1(a) + |1-a|^{\frac{5}{3}}Q_3(a)\big] + Q_2(a)
}
for all $a\ge0$. Then $Q_1, Q_2, Q_3$ are smooth away from $a=1$, $Q_2$ is 
continuous at $a=1$, $Q_1(a)=Q_3(a)=0$ for $a\ge2$, and $a^{\f13}Q_2(a)$ is 
bounded as $a\to\infty$. 
\end{cor}
\begin{proof}
For any small $q_0$ we solve \eqref{ODE7}  on $[0,1)$ which gives us $q_1,q_2$. 
We then select $(m_1,m_2)$ small so that $\tilde q_2=q_2$.
In general, we cannot expect this solution to decay faster than $a^{-\f13}$ since we will not hit the linear relation between
 $m_1$ and $m_2$ needed for this to happen. 
\end{proof}

Note that the solutions of Corollary~\ref{cor:3} are still small, since the contraction arguments by
means of which they were constructed require smallness.  This is also reflected in the property that
the nonlinearity can be both focusing and defocusing.  The smallness is expressed by the estimate
\[
|q_0|+ |\tilde q_1|\ll 1
\]
since then also $|q_1|+|q_2|\ll 1$ and $|\tilde q_2|=|q_2|\ll 1$.  

Later we shall modify the construction so as to allow large (in some sense) solutions outside of the light-cone.
For this it is essential that we only match $q_2=\tilde q_2$, since the parameter $\tilde q_1$ will be taken large. 
This construction will only be possible for the defocusing equation. 

\section{Removing the singularity on the light-cone}\label{sec:surgery1}

Departing from the singular self-similar solutions constructed above, 
we now attempt to build {\it{global smooth solutions}} to~\eqref{NLW7} which are large in a suitable sense. 
In effect, we expect them to have infinite critical norm. 
Consider the self-similar solutions constructed in the preceding section, in particular $Q_0$ from Corollary~\ref{cor:3}. 
Now set 
$$
u_0(t,r) := t^{-\f13} Q_0(r/t)
$$
which we may assume to be of class $C^0$ across the light-cone $a =\f{r}{t}= 1$, but in general no better.  
By construction, this function solves~\eqref{NLW7} away from $t=r$. 
Moreover, $|Q_0(a)|\less a^{-\f13}$ implies that 
\EQ{\label{eq:u0 dec}
|u_0(t,r) | \less (|q_{0}|+|\tilde q_{1}|) r^{-\f13} \quad \forall\; r>0
}
In view of~\eqref{Q0}, 
\EQ{\label{eq:u0}
u_0(t, r) := t^{-\frac{1}{3}}|1-a|^{\frac{2}{3}}\big[Q_1(a) + |1-a|^{\frac{5}{3}}Q_3(a)\big] + t^{-\frac{1}{3}} Q_2(a),
}
where the functions $Q_{3}$ is  expected to be discontinuous across $a = 1$, while the functions $Q_1(a),Q_2(a)$ are continuous on~$a\ge0$. In fact, writing 
\EQ{\label{Qa1}
Q_2(a) = Q_2(1) + Q_2(a) - Q_2(1),
}
we have $|Q_2(a) - Q_2(1)| = O(|1-a|)$, and it is natural to
incorporate this term into the term $$|1-a|^{\frac{2}{3}}\big[Q_1(a)+ |1-a|^{\frac{5}{3}}Q_3(a)\big]$$ in our representation of $u(t, r)$. 
Thus, with
\[
\tilde Q_3(a):= |1-a|^{\f43}Q_3(a) + |1-a|^{-1} (Q_2(a) - Q_2(1))
\]
 we obtain 
\[
u_0(t, r) = t^{-\frac{1}{3}}|1-a|^{\frac{2}{3}}\big[Q_1(a) + |1-a|^{\frac{1}{3}}\tilde Q_3(a)\big] + t^{-\frac{1}{3}}Q_2(1),
\]
where $Q_3$ is smooth away from $a= 1$ but possibly discontinuous across it.  We shall now abuse notation and just write $Q_3$ again instead of~$\tilde Q_3$. 

Thus we have now incorporated all the singular behavior of this solution into the term 
\[
 |1-a|^{\frac{2}{3}}\big[Q_1(a) + |1-a|^{\frac{1}{3}}Q_3(a)\big] =: |1-a|^{\frac{2}{3}}X(a)
\]
In order to excise the singularity, we introduce a smooth cutoff $\chi(t-r)$, which localizes the expression smoothly to a fixed distance $C$ from the light-cone, i.e.,  $|t-r|\geq C$; the constant $C$ here plays no role.  In other words, $\chi(v)=1$ for $|v|\ge 2C$ and $\chi(v)=0$ for $|v|\le C$. 

Thus we introduce the following approximate solution 
\begin{equation}
\label{eq:approxsol}
\begin{split}
u(t, r) &= t^{-\frac{1}{3}}\chi(t-r)|1-a|^{\frac{2}{3}}X(a)+ t^{-\frac{1}{3}}Q_2(1).
\end{split}
\end{equation}
Note that $u(t,r)=u_0(t,r)$ for all $|t-r|\ge 2C$. By construction, we have the following smallness property
which will play an important role in our argument: 
\EQ{
\label{eq:small u}
\| u\|_{L_t^6([T,T+1], L_x^{18}(\R^3))}\ll 1
}
uniformly in $T\ge1$. The norm here is an example of a Strichartz norm, see Lemma~\ref{lem:ST}. 

We now need to understand the error associated with the ansatz $u(t,r)$ in \eqref{eq:approxsol}, i.e.,  estimate
\[
-u_{tt} + \triangle u \mp u^7
\]
We compute 
\begin{align*}
-u_{tt} + \triangle u &\mp u^7\\
=&\chi(t-r)\big(-\partial_t^2 + \partial_r^2 + \frac{2}{r}\partial_r\big)\big(t^{-\frac{1}{3}}|1-a|^{\frac{2}{3}}X(a)+ t^{-\frac{1}{3}}Q_2(1)\big)\\
&+(1-\chi(t-r))\big(-\partial_t^2 + \partial_r^2 + \frac{2}{r}\partial_r\big)\big(t^{-\frac{1}{3}}Q_2(1)\big) + e_3\\
&\mp\big(t^{-\frac{1}{3}}\chi(t-r)|1-a|^{\frac{2}{3}}X(a)+ t^{-\frac{1}{3}}Q_2(1)\big)^7\\
=&(1-\chi(t-r))\big(-\partial_t^2 \big)\big(t^{-\frac{1}{3}}Q_2(1)\big)\\
&\pm\big[\chi(t-r)\big(t^{-\frac{1}{3}}|1-a|^{\frac{2}{3}}X(a) + t^{-\frac{1}{3}}Q_2(1)\big)^7\\
&\mp\big(t^{-\frac{1}{3}}\chi(t-r)|1-a|^{\frac{2}{3}}X(a)+ t^{-\frac{1}{3}}Q_2(1)\big)^7\big] + e_3\\
=:& e_1 + e_2 + e_3
\end{align*}
where $e_3$ denotes those terms where at least one derivative falls on $\chi(t-r)$. 
Due to the definition of $\chi$, we may include a cutoff $(1-\tilde{\chi}(t-r))$ in front of $e_2$, 
where $\tilde{\chi}$ localizes to $|t-r|\geq 2C$, i.e., we can write 
\begin{align*}
e_2 =& \pm\chi(t-r)\big(t^{-\frac{1}{3}}|1-a|^{\frac{2}{3}}X(a) + t^{-\frac{1}{3}}Q_2(1)\big)^7\\
&\mp\big(t^{-\frac{1}{3}}|1-a|^{\frac{2}{3}}Q_1(a)\chi(t-r) + t^{-\frac{1}{3}}Q_2(1)\big)^7\\
=& (1-\tilde{\chi}(t-r))\big[\pm\chi(t-r)\big(t^{-\frac{1}{3}}|1-a|^{\frac{2}{3}}X(a) + t^{-\frac{1}{3}}Q_2(1)\big)^7\\
&\hspace{2.5cm}\mp\big(t^{-\frac{1}{3}}|1-a|^{\frac{2}{3}}X(a)\chi(t-r) + t^{-\frac{1}{3}}Q_2(1)\big)^7\big]\\
\end{align*}
We can also write this as 
\[
e_2 = (1-\tilde{\chi})\big[u_1^7 - u_2^7\big],
\]
where we have the pointwise bound $$|u_1(t, r)| + |u_2(t, r)|\lesssim t^{-\frac{1}{3}}.$$  
As for $e_3$, we begin by collecting all terms in which $X(a)$ is not differentiated. Then 
with $(\ldots)'$ denoting the operator where at least one derivative falls on $\chi$, we obtain 
\begin{align*}
\big(-\partial_t^2 + \partial_r^2 &+ \frac{2}{r}\partial_r\big)'\big(t^{-\frac{1}{3}}|1-a|^{\frac{2}{3}}\chi(t-r)\big)\\
 = & 2\cdot\frac{1}{3}t^{-\frac{4}{3}}|1-a|^{\frac{2}{3}}\chi'(t-r)\\
&-2\cdot\frac{2}{3}\frac{r}{t^2}\text{sgn}(1-a)|1-a|^{-\frac{1}{3}}\chi'(t-r)t^{-\frac{1}{3}}\\
&+2\cdot\frac{2}{3}\frac{1}{t}\text{sgn}(1-a)|1-a|^{-\frac{1}{3}}\chi'(t-r)t^{-\frac{1}{3}}\\
&-\frac{2}{r}t^{-\frac{1}{3}}|1-a|^{\frac{2}{3}}\chi'(t-r)
\end{align*}
The preceding sum is seen to simplify to 
\[
2(\frac{1}{t} - \frac{1}{r})t^{-\frac{1}{3}}\frac{|t-r|^{\frac{2}{3}}}{t^{\frac{2}{3}}}\chi'(t-r) = -\frac{2(t-r)|t-r|^{\frac{2}{3}}\chi'(t-r)}{rt^2}
\]
which is one power of $t$ better than expected.  For this gain it is important that $\chi(t-r)$ solves the $1$-dimensional
wave equation. 

The  terms in $e_3$ where one derivative falls on $X(a)$ contribute
\EQ{ \nonumber 
&2 t^{-\frac{1}{3}}|1-a|^{\frac{2}{3}}X'(a)(-\p_t a  - \p_r a)\chi'(t-r) \\
&= 2 t^{-\frac{7}{3}}|1-a|^{\frac{2}{3}}X'(a)(r-t)\chi'(t-r)
}
This term is localized to the region $|r-t|\less 1$ and since 
\[
|1-a|^{\frac{2}{3}}X'(a) = Q_1'(a) + |1-a|Q_3'(a) - \f13 \sign(1-a)Q_3(a)
\]
it is of size  $t^{-\frac{7}{3}}$ on that region. 
The remaining errors $e_{1,2}$ have the same properties, i.e., they are also localized to the region $|r-t|\less 1$ and are of 
size $ t^{-\frac{7}{3}}$.  

Hence all these errors are seen to belong to  $L_t^1 L_x^2$ for $t\ge1$, since 
\[
\| t^{-\f73} (1-\chi(t-r)) \|_{L^2(\R^3)} \less t^{-\f43} \in L^1(1,\infty)
\]
Thus all these errors beat the scaling. This is an
essential feature of our construction.

\section{Completing the approximate solution to an exact one}
\label{sec:uv}

We now attempt to construct an exact solution of the form 
\[
\tilde{u}(t, r): = u(t, r) + v(t, r)
\]
where $u$ is defined in \eqref{eq:approxsol}. 
The precise theorem is as follows. 

\begin{thm} 
\label{thm:main}
Let $u_0$ be sufficiently small in the sense of Corollary~\ref{cor:3}, and 
let $u(t,r)$ be as in \eqref{eq:approxsol}. 
Then for any compactly supported radial initial data 
$$v[1] = (v_0, v_1)\in \dot{H}^{\frac{7}{6}}\cap \dot{H}^1(\R^3)\times \dot{H}^{\frac{1}{6}}\cap L^2(\R^3)$$ 
and sufficiently small with respect to the natural norm, there exists $$v\in L_t^\infty \dot{H}^{\frac{7}{6}}(\R^3)\cap L_{t, loc}^\infty\dot{H}^{1}(\R^3)\cap S$$ with 
$S$ any of the Strichartz spaces in Lemma~\ref{lem:ST}, and 
$$v_t\in L_t^\infty \dot{H}^{\frac{1}{6}}(\R^3)\cap L_{t, loc}^\infty L^2(\R^3)$$ on $[1,\infty)\times \R^3$ such that 
$\tilde{u}(t, r): = u(t, r) + v(t, r)$ solves \eqref{NLW7}. If, moreover,  
$$v[1]\in \dot{H}^s(\R^3)\times \dot{H}^{s-1}(\R^3), \;\;s>\frac{7}{6},$$ 
then also $$v[t]\in \dot{H}^s(\R^3)\times \dot{H}^{s-1}(\R^3)\quad\forall\; t\geq 1.$$
\end{thm}

\noindent The proof of Theorem~\ref{thm:main} proceeds via a bootstrap argument on the 
norm $\|v\|_{\dot{H}^{\frac{7}{6}}\cap \dot{H}^1}$.  More precisely, assuming the solution to exist on an 
interval $[1,T]$ of regularity $\dot{H}^{\frac{7}{6}}\cap \dot{H}^1(\R^3)$, we deduce an a priori bound on a 
slightly time-weighted version of the preceding norm, where the weight depends on the data, but is independent of $T$. 
Using a  local well-posedness result one can then let $T\rightarrow\infty$. 
The equation for $v$ is simply the linearized one: 
\begin{equation}\label{eq:vequation}
-v_{tt} + \triangle v \mp 7u^6v \mp \ldots \mp 7uv^6 \mp v^7 = \sum_{j=1}^3 e_j
\end{equation}
The natural space to iterate this in seems at first sight to be the Strichatz space $\|\cdot\|_{S}$ at the scaling of $\dot{H}^{\frac{7}{6}}$, which corresponds for example to the space-time norm 
\[
\|\cdot\|_{L_t^6 L_x^{18}}
\]
For the sake of completeness, let us recall a class of Strichartz estimates relevant  in this context.

\begin{lem}
\label{lem:ST}
Let $u$ be the free wave propagation of the equation  in $\R^{1+3}_{t,x}$ 
\[
\Box u = h, \quad u[0]=(f,g)
\]
where $(f,g)$ are smooth and compactly supported, and $h$ is smooth with compact support on fixed-time slices. Then
\EQ{\label{eq:ST}
&\| u\|_{L^r_t L^s_x } + \sup_t \| (u,u_t)(t) \|_{\dot{H}^{\frac{7}{6}}(\R^3)\times \dot{H}^{\frac{1}{6}}(\R^3)}  +  \| |\nabla|^{\alpha} u\|_{L^p_t L^q_x }\\
&\less \| (f,g) \|_{\dot{H}^{\frac{7}{6}}(\R^3)\times \dot{H}^{\frac{1}{6}}(\R^3)} + \| |\nabla|^{\f16} h\|_{L^1_t L^2_x}
}
where $3<r\le\infty$ and $\f{1}{3r}+\f{1}{s}=\f19$ (such as $r=6$ and $s=18$), and $2<p\le \infty$, $\f{1}{p}+\f{1}{q}=\f12$, $\alpha(q)=\f{2}{q}+\f16$.  By approximation, this extends to solutions in the Duhamel sense for which the right-hand side is finite. 
\end{lem}

However, we observe that $u$ is not bounded in $L_t^6 L_x^{18}$ due to a logarithmic divergence in infinite time. 
Thus a simple minded procedure using Strichartz and H\"older does not apply, and we are required to exploit the fine 
structure of the function~$u$. In fact, this function lives at lower and lower frequencies as $t\rightarrow\infty$. 
One may then hope to exploit some additional low-frequency control on $v$ coming from energy conservation 
to gain better control. The above theorem is a consequence of combining the following Proposition~\ref{prop:local} 
on local existence with Proposition~\ref{prop:bootstrap}, which establishes a priori control of any local solution 
to~\eqref{eq:vequation} via a bootstrap argument.  

\begin{prop}
\label{prop:local} 
Let $T\ge 1$. 
Assume that $v[T]$ is compactly supported with $$\|v[T]\|_{\dot{H}^{\frac{7}{6}}\times \dot{H}^{\frac{1}{6}}(\R^3)}\ll 1.$$ 
Then there exists  a solution $v(t)$ on the time-interval $[T,T+1]$ with the property that 
\[
v\in L_t^\infty\dot{H}^{\frac{7}{6}}([T, T+1]\times \R^3),\;\;v_t\in L_t^\infty \dot{H}^{\frac{1}{6}}([T, T+1]\times \R^3)
\]
of \eqref{eq:vequation} with compact support on every time slice $t\times \R^3$, $t\in [T, T+1]$. If 
$$v[T]\in \dot{H}^s(\R^3)\times \dot{H}^{s-1}(\R^3),$$ then also 
\[
v[t]\in \dot{H}^s(\R^3)\times \dot{H}^{s-1}(\R^3)
\]
for all $s>\frac{7}{6}$. 
\end{prop}

The proof proceeds by  a standard iteration, see Section~\ref{sec:local}. 
Taking Proposition~\ref{prop:local} for granted, the main work is then encapsulated in the following result. 

\begin{prop}\label{prop:bootstrap} 
Let $u_0$ be sufficiently small in the sense of Corollary~\ref{cor:3}, and 
let $u(t,r)$ be as in \eqref{eq:approxsol}. To be specific, in the notation of Corollary~\ref{cor:3}
we require that 
\[
|q_0|+ |\tilde q_1|\le \delta_2^3 \le \delta_{1}
\]
is small. 
Let $(v, v_t)$ be radial. Assume that 
$$v\in L_t^\infty \dot{H}^{\frac{7}{6}}\cap L_{t, loc}^\infty\dot{H}^{1}(\R^3),\quad v_t\in L_t^\infty \dot{H}^{\frac{1}{6}}\cap L_{t, loc}^\infty L^2(\R^3)$$ solves \eqref{eq:vequation} on $[1,T]\times \R^3$ (in the Duhamel sense). 
Assume further that $$\|v[1]\|_{\dot{H}^{\frac{7}{6}}\cap \dot{H}^1(\R^3)\times \dot{H}^{\frac{1}{6}}\cap L^2(\R^3)}\leq \delta_1\ll 1$$ is sufficiently small.
Then for any $C>1$ sufficiently large (in an absolute sense, independently of~$T$) with $C\delta_1\ll 1$, as well as an $\eps = \eps(\delta_2)\ll 1$, such that 
\[
\|v\|_{L_t^6 L_x^{18}([1,T]\times\R^3)} + \sup_{t\in [1,T]}\|v[t]\|_{\dot{H}^{\frac{7}{6}}\cap t^{\eps}\dot{H}^1(\R^3)\times \dot{H}^{\frac{1}{6}}\cap t^{\eps}L^2(\R^3)}\leq C\delta_1
\]
implies 
\[
\|v\|_{L_t^6 L_x^{18}([1,T]\times\R^3)} + \sup_{t\in [1,T]}\|v[t]\|_{\dot{H}^{\frac{7}{6}}\cap t^{\eps}\dot{H}^1(\R^3)\times \dot{H}^{\frac{1}{6}}\cap t^{\eps}L^2(\R^3)}\leq \frac{C}{2}\delta_1
\]
\end{prop}

The proof of this proposition is accomplished in the following two subsections. 
We shall henceforth assume that $v(t, \cdot)$ satisfies the assumptions of the proposition.  

\subsection{Energy control}\label{subsec:ener}

We note that 
\begin{align*}
&\frac{d}{dt}\int_{\R^3}\big[\frac{1}{2}\big(v_t^2 + |\nabla v|^2\big) \pm \frac{7}{2}u^6v^2 \pm\ldots\pm uv^7 \pm \frac{1}{8}v^8\big]\,dx\\
& = - \int_{\R^3} \big[\sum_{j=1}^3 e_j v_t \mp 21 u_t u^5 v^2 \mp\ldots\mp u_t v^7\big]\,dx
\end{align*}
Integrating from time $1$ to time $t$, we obtain
\begin{equation}\label{eq:enerident}\begin{split}
&\int_{\R^3}\big[\frac{1}{2}\big(v_t^2 + |\nabla v|^2\big) \pm \frac{7}{2}u^6v^2 \pm\ldots\pm uv^7 \pm \frac{1}{8}v^8\big](t, \cdot)\,dx\\
&-\int_{\R^3}\big[\frac{1}{2}\big(v_t^2 + |\nabla v|^2\big) \pm \frac{7}{2}u^6v^2 \pm\ldots+ uv^7 \pm \frac{1}{8}v^8\big](1, \cdot)\,dx\\
&= \int_1^t\int_{\R^3}\big[\sum_{j=1}^3 -e_j v_t \pm21 u_t u^5 v^2 \pm\ldots\pm u_t v^7\big]\,dx dt
\end{split}\end{equation}
Our goal is to deduce  the bound
\[
\sup_{t\in [1,T]}t^{-\eps}\|\nabla_{t,x} v(t)\|_{L_x^2}\ll C\delta_1
\]
From the estimate, cf.~\eqref{eq:u0 dec}, 
\[
\sup_{t\ge1} |u(t,r)|\less \delta_2^3\, r^{-\f13}
\]
we conclude that 
\[
\int_{\R^3} u^6 v^2\,dx\lesssim \delta_2^{18}\int_{\R^3} r^{-2}v^2\,dx\lesssim \delta_2^{18}\, \|v\|_{\dot{H}^1}^2
\]
Since $\delta_2\ll 1$, this term is thus absorbed by the principal term
\[
\int_{\R^{3}} \frac{1}{2}\big(v_t^2 + |\nabla v|^2\big)\,dx 
\]
In the defocussing case this term can be removed by positivity. 
Further, observe that for $j\in [1,5]$ the pointwise bound $$\sup_{r>0} |u(t,r)|\less \delta_2^3\, t^{-\f13} $$ implies that 
\EQ{
\label{eq:ujv8j}
\|u^jv^{8-j}\|_{L_x^1}\lesssim \delta_2\,t^{-\frac{1}{9}}\|u\|_{L_x^{9+}}^{j-\frac{1}{3}}\|v\|_{L_x^p}^{8-j}
}
where 
\[
1= \frac{j-\frac{1}{3}}{9+} + \frac{8-j}{p}
\]
This implies that 
$$\frac{81}{13}-\le p\le \f{189}{25}- \;\Longrightarrow \; 6<p<9.$$
Recall the embeddings
\[
\dot H^{1}(\R^{3}) \subset L^{6}(\R^{3}),\quad  \dot H^{\f76}(\R^{3}) \subset L^{9}(\R^{3})
\]
With  $0<\alpha<1$ determined by 
\[
\frac{1}{p} = \frac{\alpha}{6} + \frac{1-\alpha}{9}
\]
Sobolev's embedding and H\"older's inequality applied to~\eqref{eq:ujv8j} yield
\[
\|u^jv^{8-j}\|_{L_x^1}\lesssim \delta_2\, t^{-\frac{1}{9}}\|u\|_{L_x^{9+}}^{j-\frac{1}{3}}\|v\|_{\dot{H}^1}^{\alpha(8-j)}\|v\|_{\dot{H}^{\frac{7}{6}}}^{(1-\alpha)(8-j)}
\]
which we rewrite in the form
\begin{align*}
 \|u^jv^{8-j}(t, \cdot)\|_{L_x^1}\lesssim \delta_2\, t^{\eps\alpha(8-j)-\frac{1}{9}}\big(t^{-\eps}\|v(t, \cdot)\|_{\dot{H}^{1}}\big)^{\alpha(8-j)}\|v(t, \cdot)\|_{\dot{H}^{\frac{7}{6}}}^{(1-\alpha)(8-j)}
\end{align*}
If we now choose $\eps$ small enough such that 
\[
\eps\alpha(8-j)-\frac{1}{9}<0, 
\]
then we conclude that  
\begin{align*}
\sup_{t\in [1,T]}\|u^jv^{8-j}(t, \cdot)\|_{L_x^1} &\lesssim \delta_2 \big(\sup_{t\in [1,T]}t^{-\eps}\|v(t, \cdot)\|_{\dot{H}^{1}}\big)^{\alpha(8-j)}\|v(t, \cdot)\|_{\dot{H}^{\frac{7}{6}}}^{(1-\alpha)(8-j)}\\
&\lesssim \delta_2\big(C\delta_1\big)^{8-j}
\end{align*}
We also note that 
\begin{align*}
\|v^8(t, \cdot)\|_{L_x^1}&\leq \|v(t, \cdot)\|_{\dot{H}^1}^2\big(\sup_{t\in [1,T]}\|v(t, \cdot)\|_{\dot{H}^{\frac{7}{6}}}^6\big)\\
&\lesssim \|v(t, \cdot)\|_{\dot{H}^1}^2 (C\delta_1)^6
\end{align*}
where we have again used H\"older's inequality as well as the Sobolev embedding, 
and so this term can again be absorbed by the principal term 
\[
\int_{\R^{3}} \frac{1}{2}\big(v_t^2 + |\nabla v|^2\big)\,dx. 
\]
It remains to control the source terms on the right of \eqref{eq:enerident}. 
We start by estimating the contributions of the terms involving the errors~$e_j$.
First, we have
\begin{align*}
 \int_1^t\int_{\R^{3}} e_1 v_t\,dx dt =  \int_1^t\int_{\R^{3}} (1-\tilde{\chi})[u_1^7 - u_2^7] v_t\,dx dt
\end{align*}
Recall that $1-\tilde{\chi}$ localizes to the strip $|t-r|\leq 2C$. Thus since $u_{1,2}^7 = O(t^{-\frac{7}{3}})$, we infer
\[
\big\|(1-\tilde{\chi})[u_1^7 - u_2^7](t, \cdot)\big\|_{L_x^2}\lesssim \delta_2^7\, t^{-\frac{4}{3}}
\]
and so 
\begin{align*}
 \big|\int_1^t\int_{\R^{3}} e_1 v_t\,dx dt \big|&\lesssim \delta_2^7\int_1^t 
 s^{\eps-\frac{4}{3}}\,ds \big(\sup_{t\in [1,T]}t^{-\eps}\|v_t(t, \cdot)\|_{L_x^2}\big)\\
 &\lesssim \delta_2^7\, C\delta_1
\end{align*}
The contributions of the terms involving $e_{2,3}$ are handled identically. 
Next, consider the contributions of the terms 
\begin{equation}\label{eq:delicate}
 u_t u^5 v^2,\;\; u_t v^7,
\end{equation}
the intermediate terms in the space-time integral in~\eqref{eq:enerident} being handled similarly. 
The first of these terms is estimated as follows: considering the region $|t-r|\leq \frac{t}{2}$, from the formula for $u(t, r)$, we obtain 
\begin{equation}\label{eq:u_t}\begin{split}
u_t &=\big[ -\frac{1}{3}t^{-\frac{4}{3}}\chi(t-r)|1-a|^{\frac{2}{3}}X(a) -\frac{1}{3}t^{-\frac{4}{3}}Q_2(1)\big]\\
&\quad +t^{-\frac{1}{3}}\chi'(t-r)|1-a|^{\frac{2}{3}}X(a)\\
&\quad+\frac{2}{3}\frac{r}{t^2}\chi(t-r)|t-r|^{-\frac{1}{3}}\sign(t-r)X(a)\\
&\quad -t^{-\frac{1}{3}}\chi(t-r)|1-a|^{\frac{2}{3}}\frac{r}{t^2}X'(a)\\
&=: A_{1}+A_{2}+A_{3}+A_{4}
\end{split}\end{equation}
We note that (always restricting to $|r-t|\leq \frac{t}{2}$)
\[
|A_{1}| + |A_{4}|\lesssim t^{-\frac{4}{3}},
\]
and so 
\[
\int_{|t-r|\leq \frac{t}{2}}[|A_{1}|+|A_{4}|]u^5 v^2\,dx\lesssim \delta_2^6 \, t^{-1}\int_{\R^{3}} r^{-2}v^2\,dx\lesssim \delta_2^6 \,t^{-1}\|\nabla_x v\|_{L_x^2}^2
\]
For the contribution of the term $A_{3}$, we use that by radiality of $v$, we have 
\[
|v(t, r)|\lesssim r^{-\frac{1}{2}}\|v\|_{\dot{H}^1},
\]
which then gives 
\begin{align*}
\int_{|t-r|\leq \frac{t}{2}}|A_{3}|u^5 v^2\,dx&\lesssim \delta_2^6\|v\|_{\dot{H}^1}^2t^{-1}\int_{|t-r|\leq \frac{t}{2}}r^{-\frac{5}{3}}(t-r)^{-\frac{1}{3}}r^{-1}\,r^2dr\\
&\lesssim \delta_2^6\, \|v\|_{\dot{H}^1}^2t^{-1}\int_{|t-r|\leq \frac{t}{2}}r^{-\frac{2}{3}}(t-r)^{-\frac{1}{3}}dr\\
&\lesssim \delta_2^6\, \|v\|_{\dot{H}^1}^2t^{-1}
\end{align*}
Finally, for the contribution of the term $A_{2}$, we have 
\begin{align*}
\int_{|t-r|\leq \frac{t}{2}}|A_{2}|u^5 v^2\,dx&\lesssim \delta_2^6\, t^{-\frac{8}{3}}\int_{|t-r|\leq C}v^2\,dx\\
&\lesssim \delta_2^6\, t^{-\frac{4}{3}}\|v\|_{\dot{H}^1}^2
\end{align*}
where we have used H\"older's inequality and Sobolev's embedding to bound 
\[
\int_{|t-r|\leq C}v^2\,dx\lesssim t^{\frac{4}{3}}\|v\|_{\dot{H}^1}^2
\]
It follows that 
\[
\int_{|t-r|\leq \frac{t}{2}}|u_t u^5 v^2|\,dx\lesssim \delta_2^6\, t^{-1}\|v\|_{\dot{H}^1}^2
\]
On the other hand, for the region $|t-r|>\frac{t}{2}$ (assuming $t\gg 1$ as we may), we have 
\[
|u_t|\lesssim \delta_2\,  t^{-1}r^{-\frac{1}{3}},
\]
and so we conclude that 
\[
\big\| u_t u^5 v^2\big\|_{L_x^1(|t-r|>\frac{t}{2})}\lesssim \delta_2^6\, t^{-1}\int_{\R^{3}} r^{-2}v^2\,dx\lesssim \delta_2^6\, t^{-1}\|\nabla_x v\|_{L_x^2}^2
\]
where we have used Hardy's inequality in dimension $n = 3$.  It follows that 
\begin{align*}
\int_1^t\int_{\R^{3}} |u_t u^5| v^2\,dxdt&\lesssim \delta_2^6(\int_1^t s^{2\eps - 1}\,ds )\big(\sup_{t\in [1,T]}t^{-\eps}\|\nabla_{t,x}v\|_{L_x^2}\big)^2\\
&\lesssim t^{2\eps}\eps^{-1}\delta_2^6 (C\delta_1)^2
\end{align*}
For the second term in \eqref{eq:delicate} above, we have in the region $|t-r|>\frac{t}{2}$
\[
\big\|u_t v^7\big\|_{L_x^1(|t-r|>\frac{t}{2})}\lesssim 
\delta_2\, t^{-\frac{4}{3}}\|\nabla_x v\|_{L_x^2}^{4}\|\,|\nabla|^{\frac{7}{6}}v\|_{L_x^2}^{3}
\]
which gives 
\begin{align*}
\int_1^t\int_{|s-r|>\frac{s}{2}} |u_t v^7|\, dxds&
\lesssim \delta_2(C\delta_1)^3\big(\int_0^t s^{4\eps - \frac{4}{3}}\,ds\big)\; 
\big(\sup_{t\in [1,T]}t^{-\eps}\|\nabla_x v(t, \cdot)\|_{L_x^2}\big)^{4}\\
&\lesssim  \delta_2(C\delta_1)^7
\end{align*}
In the region $|t-r|\leq \frac{t}{2}$ we invoke \eqref{eq:u_t} as well as the inequality 
$$|v(t, r)|\lesssim r^{-\frac{1}{2}}\|v\|_{\dot{H}^1},$$ to obtain
\[
|u_t v^7|\lesssim  t^{-\frac{3}{2}}\|v\|_{\dot{H}^1}v^6,\quad |t-r|\leq \frac{t}{2},
\]
whence 
\begin{align*}
\int_1^t\int_{|s-r|\leq\frac{s}{2}} |u_t v^7|\, dxds&\lesssim \delta_2 \int_1^t s^{-\frac{3}{2}}s^{7\eps}\,ds\; \big(\sup_{t\in [1,T]}t^{-\eps}\|v(t, \cdot)\|_{\dot{H}^1}\big)^7\\
&\lesssim \delta_2 (C\delta_1)^6
\end{align*}
Combining the preceding bounds used to estimate the right hand side of \eqref{eq:enerident} and choosing $\delta_2 \leq\delta_1$ sufficiently small (which can be done independently of $\eps$),  we get 
\[
\sup_{t\in [1,T]}t^{-\eps}\|\nabla_{t,x} v\|_{L_x^2}\ll C\delta_1
\]
as required. 

\subsection{Critical norm control}\label{subsec:critnorm}

Here we return to \eqref{eq:vequation}, but this time we intend to control the scaling invariant norm 
\[
\|v\|_{S}: = \|v\|_{L_t^6 L_x^{18}([1,T]\times\R^3)} + \sup_{t\in [1,T]}\|\,|\nabla|^{\frac{7}{6}}v\|_{L_x^2} + \sup_{t\in [1,T]}\|\,|\nabla|^{\frac{1}{6}}\partial_tv\|_{L_x^2}
\]
From Duhamel's principle, we have 
\EQ{\label{eq:duhamel}
\big\|v\big\|_{S}&\lesssim \big\|\,|\nabla|^{\frac{1}{6}}\big(u^6 v\big)\big\|_{L_t^1 L_x^2} +\ldots +  \big\|\,|\nabla|^{\frac{1}{6}}\big(u v^6\big)\big\|_{L_t^1 L_x^2} + \|\,|\nabla|^{\frac{1}{6}}\big(v^7\big)\|_{L_t^1 L_x^2}\\
&  + \sum_{j=1}^3 \big\|\,|\nabla|^{\frac{1}{6}}e_j\big\|_{L_t^1 L_x^2} + \|v[1]\|_{\dot{H}^{\frac{7}{6}}\times \dot{H}^{\frac{1}{6}}}
}
By the explicit form of the errors $e_{j}$ derived in Section~\ref{sec:surgery1} we have
 $$\sum_{j=1}^3 \big\|\,|\nabla|^{\frac{1}{6}}e_j\big\|_{L_t^1 L_x^2} \less \delta_{2}^{3}$$ 
and by assumption $$\|v[1]\|_{\dot{H}^{\frac{7}{6}}\times \dot{H}^{\frac{1}{6}}}\le \delta_{1}$$
We now consider the more subtle terms 
\begin{equation}\label{eq:keytechnical}
\big\|\,|\nabla|^{\frac{1}{6}}\big(u^6 v\big)\big\|_{L_t^1 L_x^2},\quad \big\|\,|\nabla|^{\frac{1}{6}}\big(u v^6\big)\big\|_{L_t^1 L_x^2},
\end{equation}
the remaining intermediate power interactions being handled similarly (the term $v^{7}$ will be dealt with at the end).
The ideas involved in estimating these products are as follows:
\begin{itemize}
\item the main contribution is expected to come from the {\em diagonal interactions}, i.e., the situation in which the frequencies of all factors are about the same
\item  the factors $u$ live essentially at low frequency
\item  due to energy control, the extra derivative $\,|\nabla|^{\frac{1}{6}}$ should help us  gain from low frequencies. 
\end{itemize}
We denote the ``projection'' onto frequencies $|\xi|\le \rho$ by $P_{\le \rho}$.  As usual this is not a true
projection but rather effected by summing the Littlewood-Paley smooth frequency localizers up to that scale. In particular, we have 
$P_{\le\rho}\, f = f\ast \fy_{\rho}$, where $\fy$ is a Schwartz function with $\int \fy=1$ and $\fy_{\rho}(x)=\rho^{3}\fy(\rho x)$. 
At the expense of allowing for rapidly decaying tails in the frequency localization (which is harmless), we may also assume that $\fy$ is compactly supported. 
Thus, 
\EQ{\nn 
P_{\geq t^{-\si}}\, u (x) &= \int_{\R^{3}} (u(x) - u(x-y))\fy_{t^{-\si}}(y)\, dy \\
|P_{\geq t^{-\si}}\, u (x)| &\le  \int_{\R^{3}}  | u(x) - u(x- t^{\si} y) |  |\fy(y)|\, dy  \\
&\le  t^{\si} \int_{0}^{1} \int_{\R^{3}}  | \nabla u(x- h t^{\si} y) | |y|  |\fy(y)|\, dy  \, dh
}
which in particular implies that 
\[
\big\|P_{\geq t^{-\si}}\, u(t, \cdot)\big\|_{L_x^{18}} \less t^{\si} \| \p_{r} u(t,\cdot) \|_{L_x^{18}}
\]
Since 
\begin{align*}
u_{r}(t,r) &= -\chi'(t-r)t^{-\frac{1}{3}}(1-a)^{\frac{2}{3}}X(a) - \frac{2}{3}\chi(t-r)t^{-1}(t-r)^{-\frac{1}{3}}X(a)\\
& \quad + \chi(t-r)t^{-\frac{4}{3}}(1-a)^{\frac{2}{3}}X'(a)
\end{align*}
It follows that 
\[
\big\|P_{\geq t^{-\si}}\, u(t, \cdot)\big\|_{L_x^{18}}\lesssim \delta_2^{3}\, t^{\si+\frac{1}{9} - \frac{1}{3}}
\]
which for $\si>0$ sufficiently small is of course better than $L_t^6$. 

Returning to \eqref{eq:keytechnical} we now split 
\EQ{\label{eq:16u6v}
\,|\nabla|^{\frac{1}{6}}\big(u^6 v\big) = \,|\nabla|^{\frac{1}{6}}\big((P_{<t^{-\si}}\, u)^6 v\big) + \,|\nabla|^{\frac{1}{6}}\big(u^6 v\big)'
}
where the second term on the right-hand side is defined via this relation. 
We claim that this term can then be bounded in $L_t^1L_x^2$. 
In fact, by the fractional Leibnitz rule we can schematically estimate it at a fixed time by 
\EQ{\label{eq:schem}
& \big\|\,|\nabla|^{\frac{1}{6}}\big((P_{\geq t^{-\si}}\, u) u^5 v\big)(t, \cdot)\big\|_{L_x^2}\\
&\lesssim\big\|\,|\nabla|^{\frac{1}{6}}P_{\geq t^{-\si}}\, u(t, \cdot)\big\|_{L_x^{18}}\big\|u(t, \cdot)\big\|_{L_x^{18}}^5\big\|v(t, \cdot)\big\|_{L_x^6}\\
&\quad +\big\|P_{\geq t^{-\si}}\, u(t, \cdot)\big\|_{L_x^{18}}\big\|\,|\nabla|^{\frac{1}{6}}u(t, \cdot)\big\|_{L_x^{18}}\big\|u(t, \cdot)\big\|_{L_x^{18}}^4\big\|v(t, \cdot)\big\|_{L_x^6}\\
&\quad + \big\|P_{\geq t^{-\si}}\, u(t, \cdot)\big\|_{L_x^{18}}\big\|u(t, \cdot)\big\|_{L_x^{18}}^5\big\|\,|\nabla|^{\frac{1}{6}}v(t, \cdot)\big\|_{L_x^6}
}
To estimate the $L_t^1$-norm of the right hand sides, we use the energy bound derived previously: 
\begin{align*}
&\big\|\big\|\,|\nabla|^{\frac{1}{6}}P_{\geq t^{-\si}}\, u(t, \cdot)\big\|_{L_x^{18}}\big\|u(t, \cdot)\big\|_{L_x^{18}}^5\big\|v(t, \cdot)\big\|_{L_x^6}\big\|_{L_t^1[1,T]}\\
&\lesssim \delta_2^{3}\big\|t^{\sigma+\frac{1}{9}+\eps-\frac{1}{3}}\big\|_{L_t^{6-}}\big\|u(t, \cdot)\big\|_{L_t^{6+}L_x^{18}}^5\big\|t^{-\eps}v(t, \cdot)\big\|_{L_t^\infty L_x^6([1,T]\times\R^3)}\\
&\lesssim \delta_2^{18} C\delta_1
\end{align*}
which is $\ll C\delta_1$. The remaining terms in~\eqref{eq:schem} above are handled similarly, and so we have
reduced ourselves to estimating the first term in~\eqref{eq:16u6v}: 
\[
\big\|\,|\nabla|^{\frac{1}{6}}\big((P_{<t^{-\si}}\,u)^6 v\big)(t, \cdot)\big\|_{L_t^1 L_x^2([1,T]\times\R^3)}
\]
Using the fractional Leibnitz rule, we bound this by 
\begin{equation}\label{eq:finaltechnical}\begin{split}
&\big\|\,|\nabla|^{\frac{1}{6}}\big((P_{<t^{-\si}}\, u)^6 v\big)(t, \cdot)\big\|_{L_t^1 L_x^2([1,T]\times\R^3)}\\
&\lesssim \big\|t^{\eps}\|\,|\nabla|^{\frac{1}{6}}P_{<t^{-\si}}\, u(t, \cdot)\|_{L_x^{18}}\|P_{<t^{-\si}}u(t, \cdot)\|_{L_x^{18}}^5\big\|_{L_t^1}\big\|t^{-\eps}\|v(t, \cdot)\|_{\dot{H}^1}\big\|_{L_t^\infty}\\
&\quad +\big\| \|P_{<t^{-\si}}\, u(t, \cdot)\|_{L_x^{18+}}^6\|\,|\nabla|^{\frac{1}{6}}v(t, \cdot)\|_{L_x^{6-}}\big\|_{L_t^1}
\end{split}\end{equation}
To estimate the first term on the right, we use that 
\[
\|\,|\nabla|^{\frac{1}{6}}P_{<t^{-\si}}\, u(t, \cdot)\|_{L_x^{18}}\lesssim t^{-\frac{\sigma}{6}}\|u(t, \cdot)\|_{L_x^{18}}
\]
and so we get 
\begin{align*}
& \big\|t^{\eps}\|\,|\nabla|^{\frac{1}{6}}P_{<t^{-\si}}\,u(t, \cdot)\|_{L_x^{18}}\|P_{<t^{-\si}}\, u(t, \cdot)\|_{L_x^{18}}^5\big\|_{L_t^1}\big\|t^{-\eps}\|v(t, \cdot)\|_{\dot{H}^1}\big\|_{L_t^\infty}\\
&\lesssim \big\|t^{\frac{\eps}{6}-\frac{\si}{36}}\|u(t, \cdot)\|_{L_x^{18}}\big\|_{L_t^6}^6\big\|t^{-\eps}\|v(t, \cdot)\|_{\dot{H}^1}\big\|_{L_t^\infty}\\
&\lesssim \delta_2^{18} C\delta_1
\end{align*}
For the second term on the right in \eqref{eq:finaltechnical}, the idea is that we can place $ \,|\nabla|^{\frac{1}{6}}v$ into $L_x^{6-}$ while paying a small power of $t$, while placing the low frequency factors into $L_x^{18+}$, gaining a bit in $t^{-1}$. 
Specifically, from Sobolev's embedding $$\dot{H}^{\frac{5}{6}}(\R^{3})\subset L^{\frac{9}{2}}(\R^{3}),$$ we infer that 
\[
\big\|\,|\nabla|^{\frac{1}{6}}v(t, \cdot)\big\|_{L_x^{\frac{9}{2}}}\lesssim t^{\eps}\big(t^{-\eps}\|\nabla_{t,x}v(t, \cdot)\|_{L_x^2}\big),
\]
while Bernstein's inequality implies that
\[
\big\|P_{<t^{-\si}}\, u(t, \cdot)\big\|_{L_x^{\frac{108}{5}}}\lesssim t^{-\frac{\si}{36}}\|u(t, \cdot)\|_{L_x^{18}}
\]
Since the $\eps$ can be made small independently of the $\si$ above, and in particular we may assume $\eps<\frac{\si}{6}$, we then get 
\begin{align*}
\big\|(P_{<t^{-\si}}\, u)^6 \,|\nabla|^{\frac{1}{6}}v\big\|_{L_t^1 L_x^2([1,T]\times\R^3)}
&\lesssim \big\|t^{\frac{\eps}{6}}P_{<t^{-\si}}\, u\big\|_{L_t^{6}L_x^{\frac{108}{5}}}^6\big\|t^{-\eps}\,|\nabla|^{\frac{1}{6}}v\big\|_{L_t^{\infty}L_x^{\frac{9}{2}}}\\
&\lesssim \delta_2^{18}\sup_{t}\big(t^{-\eps}\|\nabla_{t,x}v(t, \cdot)\|_{L_x^2}\big)\\
&\lesssim C\delta_1 \delta_2^{18}
\end{align*}
The second term in \eqref{eq:keytechnical} is handled similarly: we split
\begin{equation}\label{eq:lastterm}\begin{split}
\|\,|\nabla|^{\frac{1}{6}}\big(uv^6\big)\|_{L_t^1 L_x^2([1,T]\times \R^3)}&\leq \|\,|\nabla|^{\frac{1}{6}}\big((P_{\geq t^{-\si}}\,u)v^6\big)\|_{L_t^1 L_x^2([1,T]\times \R^3)}\\
&\quad+\|\,|\nabla|^{\frac{1}{6}}\big((P_{<t^{-\si}}\,u)v^6\big)\|_{L_t^1 L_x^2([1,T]\times \R^3)}
\end{split}\end{equation}
The first term on the right-hand side  is estimated by 
\begin{align*}
&\|\,|\nabla|^{\frac{1}{6}}\big((P_{\geq t^{-\si}}\,u)v^6\big)\|_{L_t^1 L_x^2([1,T]\times \R^3)}\\
&\lesssim\big\| t^{\eps}\|\,|\nabla|^{\frac{1}{6}}P_{\geq t^{-\si}}\,u\|_{L_x^{18}}\|v\|_{L_x^{18}}^5\big\|_{L_t^1([1,T])}\sup_{t\in [1,T]}t^{-\eps}\|v(t, \cdot)\|_{\dot{H}^1}\\
&\quad +\big\|\|P_{\geq t^{-\si}}u\|_{L_x^{18}}\|v\|_{L_x^{18}}^5\big\|_{L_t^1([1,T])}\sup_{t\in [1,T]}\|\,|\nabla|^{\frac{1}{6}}v(t, \cdot)\|_{\dot{H}^1}\\
&\lesssim \delta_2^{3}\|t^{\eps+\si+\frac{1}{9}-\frac{1}{3}}\|_{L_t^6([1,T])}\|v\|_{L_t^6 L_x^{18}}^5 \sup_{t\in [1,T]}t^{-\eps}\|v(t, \cdot)\|_{\dot{H}^1}\\
&\quad + \delta_2^{3}\|t^{\si+\frac{1}{9}-\frac{1}{3}}\|_{L_t^6([1,T])}\|v\|_{L_t^6 L_x^{18}}^5 \sup_{t\in [1,T]}\|v(t, \cdot)\|_{\dot{H}^{\frac{7}{6}}},
\end{align*}
and so we can bound the last two terms by $\lesssim \delta_2^{3} (C\delta_1)^6$. 
To handle the second term on the right-hand side of~\eqref{eq:lastterm}, we have 
\begin{align*}
&\|\,|\nabla|^{\frac{1}{6}}\big((P_{<t^{-\si}}\,u)v^6\big)\|_{L_t^1 L_x^2([1,T]\times \R^3)}\\
&\lesssim \|t^{\eps}\,|\nabla|^{\frac{1}{6}}P_{<t^{-\si}}\, u\|_{L_t^6 L_x^{18}([1,T]\times \R^3)}\|v\|_{L_t^6 L_x^{18}([1,T]\times \R^3)}^5\sup_{t\in [1,T]}t^{-\eps}\|v(t, \cdot)\|_{\dot{H}^{1}}\\
&\quad +\| t^{\eps}P_{<t^{-\si}}\,u\|_{L_t^{6} L_x^{\infty}([1,T]\times \R^3)}\|v\|_{L_t^6 L_x^{18}([1,T]\times \R^3)}^5
\sup_{t\in [1,T]}t^{-\eps}\|\,|\nabla|^{\frac{1}{6}}v(t, \cdot)\|_{L_x^{\frac{9}{2}}}
\end{align*}
The first product on the right-hand side can be estimated by 
\[
\lesssim \delta_2^{3}\|t^{\eps - \frac{\si}{6}-\frac{1}{6}}\|_{L_t^6([1,T]\times \R^3)}(C\delta_1)^6\lesssim \delta_2^{3}(C\delta_1)^6
\]
For the second term above, we infer  from Bernstein's inequality that 
\[
\| t^{\eps}P_{<t^{-\si}}\,u\|_{L_t^{6} L_x^{\infty}([1,T]\times \R^3)}\lesssim \delta_2^{3}\|t^{\eps-\frac{\si}{6}-\frac{1}{6}}\|_{L_t^6([1,T])}\lesssim \delta_2^{3},
\]
and so we obtain 
\begin{align*}
&\| t^{\eps}P_{<t^{-\si}}\,u\|_{L_t^{6} L_x^{\infty}([1,T]\times \R^3)}\|v\|_{L_t^6 L_x^{18}([1,T]\times \R^3)}^5\sup_{t\in [1,T]}t^{-\eps}\|\,|\nabla|^{\frac{1}{6}}v(t, \cdot)\|_{L_x^{\frac{9}{2}}}\\
&\lesssim \delta_2^{3} (C\delta_1)^6
\end{align*}
This concludes the estimate for the second term in \eqref{eq:keytechnical}.

To complete the bootstrap for the critical Strichartz norm, we also need to bound the contribution of the pure power term $v^7$ in \eqref{eq:duhamel}. This we do by 
\[
\|\,|\nabla|^{\frac{1}{6}}(v^7)\|_{L_t^1 L_x^2([1,T]\times \R^3)}\lesssim \|v\|_{L_t^6 L_x^{18}([1,T]\times \R^3)}^6\|v\|_{L_t^\infty\dot{H}^{\frac{7}{6}}([1,T]\times \R^3)}\lesssim (C\delta_1)^7
\]
All of the preceding bounds are $\ll C\delta_1$ provided we pick $\delta_1 \geq \delta_2^{3}$ sufficiently small, which completes the bootstrap and hence the proof of the proposition. 

\subsection{The proofs of Theorems~\ref{thm:main}, \ref{thm:Main1}, \ref{thm:intro1}  }

Theorem~\ref{thm:main} follows from Proposition~\ref{prop:bootstrap} by the standard bootstrap argument; indeed, we may initially take the constant $C$ 
as large as we like, depending on the solution itself. The finiteness of the constant being guaranteed by the local well-posedness as in Proposition~\ref{prop:local}. 
Then the constant can be lowered until it reaches some large but absolute size independent of the time of existence. 

Theorem~\ref{thm:Main1} follows by taking the solution $u+v$ to \eqref{NLW7} constructed in Theorem~\ref{thm:main}. The data $(f,g)$ are equal to  
$((u+v)(1,\cdot), (u+v)_t(1,\cdot))$. The infinite critical norm being a consequence of the fact that $(v,v_t)$ has finite critical norm, but $(u,u_t)$ being
given by~\eqref{eq:approxsol}. The finiteness of $\dot H^s\times \dot H^{s-1}$ for $s>\f76$ is a result of the asymptotic decay of $|u(t,r)|\sim r^{-\f13}$ (or $|u(t,r)|\sim r^{-\f43}$ non-generically) 
and $|u_t(t,r)|\sim r^{-\f43}$ as $r\to\infty$, respectively. So $u$ lies in these spaces, and the perturbation $v$ does so by construction. The stability claimed
by the theorem is a result of the fact that the perturbation $v$ belongs to an open set in the norms of Theorem~\ref{thm:main}. 

Theorem~\ref{thm:intro1} follows from Theorem~\ref{thm:main} by truncation. Indeed, given $M$ as in~\eqref{gross} we choose $R$ so large that
the data $(f,g)$ as in Theorem~\ref{thm:Main1} have critical norm exceeding  $M$ or any other large constant when restricted to $\{|x|<R\}$. The theorem 
then follows by finite propagation speed, rescaling, and the fact that we may make the Strichartz norms of $u$ large provided we integrate over a sufficiently
large time-interval. For this theorem it is essential to note that blowup for~\eqref{NLW7} can only occur at the origin, since we are dealing
with the radial problem and there is a pointwise a priori bound for $r>0$ for all times $t>0$ in the defocusing case as a result of the Strauss' estimate
and the positive definite conserved energy for the defocusing equation~\eqref{NLW7}.  This is the  reason why we restrict to the defocusing equation here.

\section{Larger global solutions in the defocussing case}
\label{sec:12}

In this section we revisit the ODE theory from Section~\ref{sec:selfsim} in the defocusing case.  More precisely, 
we wish to exploit the flexibility of Corollary~\ref{cor:3} with regard to the choice 
of the parameter~$\tilde q_{1}$. While we matched the outside solution 
with the inside one through the connection condition $\tilde q_{2}=q_{2}$ which ensures continuity, the choice of $\tilde q_{1}$
is arbitrary. In Corollary~\ref{cor:3} we still require $\tilde q_{1}$ to be small, since at that point we had
only constructed solutions in $a>1$ assuming both $\tilde q_{1}$ and $\tilde q_{2}$ to be small. 

We shall now proceed to show that solutions to the ODE~\eqref{ODE7} exist in $a>1$ for small $\tilde q_{2}$,
but large $\tilde q_{1}$.
We start by proving an analogue of Lemma~\ref{lem:near2} near $a=1$. We shall then extend the solution to 
all of $a>1$, which depends crucially on the defocusing character of the equation.  For technical reasons,
the expression in~\eqref{anear10} differs from the one in Lemma~\ref{lem:near2}.  To be specific, instead of
the power $(a-1)^{\f73}$ we use $(a-1)^{\f43}$, absorbing the factor $(a-1)$ into~$\tilde Q_3$. 

\begin{lem}
\label{lem:near10}
There exists $\eps>0$ small with the following property: 
Given $\tilde q_2\in (-\eps,\eps)$ and $\tilde{q}_1\ge1$ arbitrary, there exists a unique solution $Q(a)$ of \eqref{ODE7}  
on $[1,1+\ell ]$ for $\ell = c|\tilde q_2| \tilde{q}_1^{-\frac{4}{3}}$ for some absolute sufficiently 
small constant $c>0$ of the form
\EQ{ \label{anear10}
Q(a) & =  (a-1)^{\f23} \tilde Q_1(a) +  \tilde Q_2(a) + (a-1)^{\f43} \tilde Q_3(a)
}
with $\tilde Q_1, \tilde Q_2, \tilde Q_3\in C^\infty([1,1+\ell ])$ and 
\EQ{\label{eq:til q O}
\tilde Q_1(a) &= \tilde q_1(1+O(a-1)),\;\; \tilde Q_2(a)= \tilde q_2(1+O(a-1)),\\
\tilde Q_{3}(a) &=\tilde{q}_1O(1)
}
where the $O(\cdot)$ terms are $C^\infty$ functions in~$a\in [1,1+\ell ]$. We also have the bounds 
\EQ{\label{eq:til Q sch}
|\tilde Q_1(a)|\geq \frac{\tilde{q}_1}{2},\quad |\tilde Q_2(a)|\leq C|\tilde q_2|,\quad |\tilde Q_3(a)|\leq C\tilde{q}_1
}
uniformly in $a\in [1,1+\ell ]$ and with some absolute constant $C$. 
Finally, there exists  $a_{*}\in (1, 1+\ell ]$ so that 
\EQ{\label{Q gross}
|Q(a_{*})|\simeq |\tilde q_2|^{\f23} \tilde{q}_1^{\frac{1}{9}}
}
In particular, this can be made arbitrarily large by making  $\tilde{q}_1$ sufficiently large. 
\end{lem}
\begin{proof} We refer to the proof of Lemma~\ref{lem:near1}. We start from the representation  
\[
Q(a) = (a-1)^{\frac{2}{3}}\tilde{Q}_1(a) + \tilde{Q}_2(a) + (a-1)^{\frac{4}{3}}\tilde{Q}_3(a)
\]
where we furthermore assume the structure 
\EQ{ \label{eq:Q1Q2Q3}
\tilde{Q}_1(a) &= \tilde{q}_1(1+O(a-1))\\
\tilde{Q}_2(a) & = \tilde{q}_2(1+O(a-1)) \\
 \tilde{Q}_3(a) &= \tilde{q}_1O(1)
}
We then obtain $\tilde{Q}_{1,2,3}(a)$ as fixed points of the following system, see \eqref{drei sys},  
\EQ{\label{drei sys tilde}
\tilde{Q}_1(a) &=   \tilde{q}_1 a^{-1}  + (a-1)^{-\f23}  \int_1^a G(a,b) (b-1)^{\f23} N_1(b,\tilde{Q}_1, \tilde{Q}_2,\tilde{Q}_3) \, db \\
\tilde{Q}_2(a) &= \tilde{q}_2 2^{-\f23} \fy_2(a) +\int_1^a G(a,b) N_0(b,\tilde{Q}_1, \tilde{Q}_2,\tilde{Q}_3)\, db \\
\tilde{Q}_3(a) &=  (a-1)^{-\f43} \int_1^a G(a,b)  (b-1)^{\f43} N_2(b,\tilde{Q}_1, \tilde{Q}_2,\tilde{Q}_3)\, db
}
The Green function is  the one from \eqref{Gab2}, viz. 
\EQ{\nn 
G(a,b)= g_1(a,b) + (b-1)^{\f23} (a-1)^{-\f23} g_2(a,b)
}
and the source functions $N_k(b,\tilde{Q}_1, \tilde{Q}_2,\tilde{Q}_3)$, $k = 0,1,2$, can be written schematically as follows: 
\EQ{\nn
N_1(b,\tilde{Q}_1, \tilde{Q}_2,\tilde{Q}_3) &= \!\!\! \sum_{2m_1 + 4m_3\cong 2(3)} \!\!\!   C_{m_1,m_2,m_3}\, (b-1)^{\frac{2m_1 + 4m_3 - 2}{3}}\tilde{Q}_1^{m_1}(b)\tilde{Q}_2^{m_2}(b)\tilde{Q}_3^{m_3}(b) \\
N_0(b,\tilde{Q}_1, \tilde{Q}_2,\tilde{Q}_3) &=  \!\!\!  \sum_{2m_1 + 4m_3\cong 0(3)} \!\!\!   C_{m_1,m_2,m_3}\, (b-1)^{\frac{2m_1 + 4m_3}{3}}\tilde{Q}_1^{m_1}(b)\tilde{Q}_2^{m_2}(b)\tilde{Q}_3^{m_3}(b)\\
N_2(b,\tilde{Q}_1, \tilde{Q}_2,\tilde{Q}_3) &=  \!\!\!  \sum_{2m_1 + 4m_3\cong 1(3)} \!\!\!   C_{m_1,m_2,m_3}\, (b-1)^{\frac{2m_1 + 4m_3-4}{3}}\tilde{Q}_1^{m_1}(b)\tilde{Q}_2^{m_2}(b)\tilde{Q}_3^{m_3}(b)
}
In these sums, $m_1,m_2,m_3$ are nonnegative integers such that
\[
m_1+m_2+m_3=7.
\]
In the first sum, we require a further restriction in the form $m_1+m_3\ge1$,
and in the third $m_1\ge2$ or $m_3\ge1$. 

To obtain the desired fixed point for \eqref{drei sys tilde}, we show that the bounds 
\EQ{\label{Qtil bd}
|\tilde Q_1(a)| + |\tilde Q_3(a)|\leq C\tilde{q}_1, \quad |\tilde Q_2(a)|\leq C|\tilde{q}_2|
}
improve upon themselves on the interval $a\in [1, 1+\ell]$ with $\ell$ as in the statement of the lemma, once they are
inserted into the system~\eqref{drei sys tilde}.  To be precise, we will prove that~\eqref{Qtil bd}
implies the same bounds with $C/2$ instead of~$C$ provided that constant is bigger than some absolute one. 

To accomplish this, we shall rely on  the choice of $$\ell = c|\tilde q_2| \tilde{q}_1^{-\frac{4}{3}}.$$
Indeed, in the $\tilde Q_1$ integral we estimate 
\begin{align*}
&\tilde{q}_1^{-1}\big|(a-1)^{-\f23}  \int_1^a G(a,b) (b-1)^{\f23} N_1(b,\tilde{Q}_1, \tilde{Q}_2,\tilde{Q}_3) \, db \big|\\
&\lesssim  \sum_{2m_1 + 4m_3\cong 2(3)} \ell\, \ell^{\frac{2 (m_1-1)}{3}}\tilde{q}_1^{m_1-1}\ell^{\frac{4m_3}{3}}\tilde{q}_1^{m_3}
\end{align*}
Recall that we have $m_1+m_3\ge1$ in this case.
First note that 
\[
\ell^{\frac{4m_3}{3}}\tilde{q}_1^{m_3}\le  1
\]
for all $m_3\ge0$.  If $m_1\ge1$, then  we may estimate 
\[
\ell\,  \ell^{\frac{2 (m_1-1)}{3}}\tilde{q}_1^{m_1-1}\leq \ell  + \ell^5 \tilde q_1^6 \ll 1
\]
for all $ m_1 = 1,2,\ldots,7$. If, on the other hand, $m_1=0$ then either $m_3=2$ or $m_3=5$ whence 
\[
 \ell\, \ell^{\frac{2 (m_1-1)}{3}}\tilde{q}_1^{m_1-1}\ell^{\frac{4m_3}{3}}\tilde{q}_1^{m_3} 
 \less  \ell^3 \tilde q_1 + \ell^7 \tilde q_1^4 \ll 1
\]
As for the source term involving $N_0$, we get 
\begin{align*}
&\big|\int_1^a G(a,b) N_0(b,\tilde{Q}_1, \tilde{Q}_2,\tilde{Q}_3)\, db\big|\\
&\lesssim \sum_{2m_1 + 4m_3\cong 0(3)}\ell\, \ell^{\frac{2m_1 + 4m_3}{3}}\tilde{q}_1^{m_1 + m_3}
\end{align*}
where the sum runs over all integers $0\le m_1, m_3\le 7$. Note that here we necessarily have $m_1\leq 6$. The general term 
of this finite sum is decreasing in $m_3$ irrespective of $m_1$. So it suffices to set $m_3=0$ and to evaluate at the endpoints  $m_1=0$
and $m_1=6$, respectively.  In summary this yields the bound
\[
\ell + (\ell^{\f56} \tilde q_1)^6 \ll |\tilde q_2| 
\]
Finally, for the contribution of $N_3$, we get in case $m_3\geq 1$
\begin{align*}
&\tilde{q}_1^{-1}\big|(a-1)^{-\f43} \int_1^a G(a,b)  (b-1)^{\f43} N_2(b,\tilde{Q}_1, \tilde{Q}_2,\tilde{Q}_3)\, db\big|\\
&\lesssim \sum_{2m_1 + 4m_3\cong 1(3)} \ell\, \ell^{\frac{2m_1 + 4(m_3 - 1)}{3}}\tilde{q}_1^{m_1 + m_3-1}
\end{align*}
Once again, the general term is decreasing in $m_3$.  So it suffices to consider the pairs $(m_1, m_3)$ from the following list
\[
(0,1),\; (1,2),\;  (2,0),\; (3,1),\; (4,2),\; (5,0),\; (6,1)
\]
in which $m_3$ is always the smallest possible one given the value of $m_1$. 
Hence the upper bound is of the form
\EQ{\nn
&\less \ell(1+\ell \tilde q_1) + \ell \tilde q_1\big[1 + (\ell \tilde q_1)^2 + (\ell \tilde q_1)^4\big] + \ell^3 \tilde q_1^4\big[1 + (\ell \tilde q_1)^2\big]\\
&\less \ell + \ell \tilde q_1 + \ell^3 \tilde q_1^4 \ll 1
}
where we have used that  $\ell \tilde q_1 = c|\tilde q_2| \tilde q_1^{-\f13}\ll 1$.

The existence of the desired fixed point on $[1,1+\ell]$ now follows from this in a standard fashion,
 as well as $\tilde{Q}_1(a)\geq \frac{\tilde{q}_1}{2}$ provided we pick the constant $c$ small enough. 
 To be specific, we define the space 
 \[
 X:=\{ (\tilde Q_{1}, \tilde Q_{2}, \tilde Q_{3}) \in (C^{0}([1,1+\ell]))^{3} \:|\: \;\;   \eqref{Qtil bd} \text{\ \ holds\ }\}
 \]
 where the constant in~\eqref{Qtil bd} is an absolute one. By the preceding analysis, we see that the complete 
 metric space~$X$ is mapped onto itself by~\eqref{drei sys tilde}. Moreover, taking differences shows that
 the system is a contraction in~$X$. It is a standard calculus exercise to verify that the integrals in~\eqref{drei sys tilde}
 are $C^{1}([1,1+\ell])$, and iterating this property shows that the left-hand side of~\eqref{drei sys tilde} is in fact~$C^{\infty}([1,1+\ell])$.
In particular,  we obtain~\eqref{eq:til q O} and~\eqref{eq:til Q sch}, the latter being implied by the integral
estimates from above. 
 
Finally, picking $a_* = 1+\frac{\ell}{2}$, we obtain
\[
Q(a_*) \simeq \big(c|\tilde q_2|\tilde{q}_1^{-\frac{4}{3}}\big)^{\frac{2}{3}}\tilde{q}_1 \simeq |\tilde q_2|^{\f23} \tilde{q}_1^{\frac{1}{9}}
\]
which gives \eqref{Q gross}. 
\end{proof}

Having solved the ODE~\eqref{ODE7}  near the singularity, we shall now show that the solution
may be extended to the region $a\geq 1+\ell$.  For this we need the equation to be defocusing.

\begin{lem}
\label{lem:farfromcone}
The solutions to~\eqref{ODE7}
 constructed in the preceding lemma on $(1, 1+\ell]$ can be extended to $(1,\infty)$ as smooth globally
 bounded functions~$Q(a)$.
 For large values we have the asymptotics $$|Q(a)|\less  a^{-\f13} ,\qquad |Q'(a)|\less a^{-\f43}$$
 as $a\to\infty$ with non vanishing constants $c_{1}, c_{2}$. 
\end{lem}
\begin{proof}
We construct an integrating factor to remove the first order derivative in~\eqref{ODE7}. 
Thus introduce the auxiliary function 
\[
f(\tilde{a}) = \frac{1}{2}\frac{\frac{8}{3}\tilde{a} - \frac{2}{\tilde{a}}}{\tilde{a}^2 - 1},\quad \tilde{a}\in [1+\ell, \infty)
\]
as well as the new dependent variable 
\[
X(a): = Q(a)w(a),\qquad w(a):= e^{\int_{1+\ell}^af(\tilde{a})\,d\tilde{a}}
\]
Then the original ODE is equivalent to the following one for $X$: 
\EQ{\label{eq:nonlin osc}
X''(a) + g(a) X + \big(\frac{e^{-6\int_{1+\ell}^a f(\tilde{a})\,d\tilde{a}}}{a^2-1}\big)X^7 = 0
}
where 
\[
g(a) := \frac{5}{9(a^2-1)^2}
\]
To obtain global regularity, it suffices to exhibit an a priori $L^\infty$-bound on any finite interval $[1+\ell, L]$ for $X$. 

In order to obtain such a  bound, we multiply the equation by $X'$ and integrate. This yields an ``energy estimate''
\begin{align*}
&\frac{1}{2}(X')^2(a) + \frac{1}{2}X^2(a)g(a) + \frac{1}{8}\big(\frac{e^{-6\int_{1+\ell}^a f(\tilde{a})\,d\tilde{a}}}{a^2-1}\big)X^8(a)\\
&= -\int_{1+\ell}^a\big(\frac{3}{4}f(\tilde{a}) + \frac{a}{4(a^2-1)^2}\big)e^{-6\int_{1+\ell}^{\tilde{a}} f(a_1)\,da_1}X(a)^8\,d\tilde{a}\\
&-\int_{1+\ell}^a \frac{10a}{9(a^2-1)^3}X^2(\tilde{a})\,d\tilde{a} + \frac{1}{2}(X')^2(1+\ell) + \frac{1}{2}X^2(1+\ell)g(1+\ell)\\
&+ \frac{1}{8(\ell^2+2\ell)}X^8(1+\ell)\\
&\le \frac{1}{2}(X')^2(1+\ell) + \frac{1}{2}X^2(1+\ell)g(1+\ell)+ \frac{1}{8(\ell^2+2\ell)}X^8(1+\ell)
\end{align*} 
Here we used that $f(a)>0$ for all $a>1$. Thus,  one has an a priori bound 
\[
\big(\frac{e^{-6\int_{1+\ell}^a f(\tilde{a})\,d\tilde{a}}}{a^2-1}\big)X^8(a)\leq C(\ell),\qquad a\in [1+\ell, \infty), 
\]
Since $$\int_{1+\ell}^a f(\tilde{a})\,d\tilde{a}\sim \frac{4}{3}\log a,\qquad w(a)\sim a^{\f43}$$ as $a\rightarrow\infty$,  we obtain
\[
\frac{e^{-6\int_{1+\ell}^a f(\tilde{a})\,d\tilde{a}}}{a^2-1} \sim a^{-10}
\]
and hence 
\[
|X(a)|\leq D(\ell)a^{\frac{5}{4}},\qquad a\in [1+\ell, \infty)
\]
Returning to the original dependent variable this implies the decay
\[
|Q(a)| = |X(a)|e^{-\int_{1+\ell}^a f(\tilde{a})\,d\tilde{a}}\leq E(\ell)a^{-\frac{1}{12}} 
\]
as $a\rightarrow\infty$, whence $Q(a)$ converges to zero as $a\rightarrow\infty$. 
From the energy estimate we furthermore infer that $|X'(a)|\le C(\ell)$ for all $a\ge 1+\ell$
whence
\EQ{\nn
& |Q'(a)w(a)+Q(a)w'(a)| \le C(\ell) \\
& |Q'(a)| \le C(\ell)w^{-1}(a) + |Q(a)||w'(a)|w^{-1}(a) \less a^{-\f{13}{12}}
}
But this implies that for any $\eps>0$, there exists a  $a_\eps$ sufficiently large such that 
\[
|Q(a_{\eps})| + |Q'(a_{\eps})|<\eps
\]
This means that on the interval $[a_{\eps}, \infty)$ we are in the small data case, and we may solve  \eqref{ODE7} by perturbing around the corresponding solution of the linear part.  

To be specific, we are precisely in the regime of Lemma~\ref{lem:largea}. 
For the reader's convenience, we sketch the details. With the linear fundamental system
\EQ{\nn
\tilde \fy_1(a)=a^{-1}(a-1)^{\f23},\quad  \fy_1(a)=a^{-1}(1+a)^{\f23}
}
we define the Green function
\EQ{\nn 
G(a,b) &:= \frac{\tilde\fy_1(a)\fy_2(b)- \tilde\fy_1(b)\fy_2(a)}{W(b) (b^2-1)} \\
W(a) &:=  \tilde \fy_1(a)\fy_2'(a) - \tilde \fy_1'(a)\fy_2(a)
}
The denominator is $W(b)(b^2-1)$, which decays at the rate~$b^{-\f23}$ as $b\to\infty$. 
The perturbative  approach is to seek a nonlinear solution in the form
\EQ{\label{m1m2'}
\tilde Q(a) = m_1 \tilde \fy_1(a) + m_2 \fy_2(a) + \int_a^\infty G(a,b) \tilde Q(b)^7\, db
}
for all $a\ge a_{\eps}$ where $m_{1},m_{2}$ are small. As in  Section~\ref{sec:selfsim} 
one shows that \eqref{m1m2'} admits a unique solution for any such choice of $m_{1}, m_{2}$
and that, moreover, the map $$(m_{1},m_{2})\mapsto (\tilde Q(a_{\eps}),\tilde Q'(a_{\eps}))$$
is a diffeomorphism form one small neighborhood of the origin to another. So we in particular
find $(m_{1},m_{2})$ in~\eqref{m1m2'} so that 
\[
(\tilde Q(a_{\eps}),\tilde Q'(a_{\eps})) = ( Q(a_{\eps}),Q'(a_{\eps}))
\]
and we see that $\tilde Q(a)=Q(a)$ for all $a\ge a_{\eps}$. 
This means that in fact generically we have 
\[
|Q(a)|\sim a^{-\frac{1}{3}},\,|Q'(a)| \sim a^{-\frac{4}{3}}
\]
as $a\rightarrow\infty$, using the asymptotics of the fundamental system $\tilde{\fy}_1(a), \fy_2(a)$. 
But it  is possible that we satisfy the linear relation between $m_{1},m_{2}$ which cancels the leading order $a^{-\f13}$, leading
the faster decay $a^{-\f43}$. 
\end{proof}

Due to the energy estimate which played a pivotal role in the proof,
the previous lemma essentially depends on the defocussing character of the problem. 
We remark that in contrast to the small solutions constructed in Section~\ref{sec:selfsim}, 
the large solutions constructed in this section may oscillate wildly in an interval $(1,a_{*})$
because equation~\eqref{eq:nonlin osc} is a nonlinear oscillator equation.

\section{Gluing the self-similar solutions, excision and completion to a global smooth solution}

In this section we follow the scheme that we deployed above  in the small solution case to  
excise the singularity from the light-cone so as the obtain global smooth solution of the defocusing
equation~\eqref{NLW7}. 
By combining Corollary~\ref{cor:3} with the results of the previous section, we arrive at the following 
conclusion.

\begin{prop}\label{prop:largeselfsimilar} Given $q_0$ small enough as well as $\tilde{q}_1$ arbitrary,
 there exists a continuous function $Q(a)$ which is smooth away from $a = 1$ and which solves \eqref{ODE7} on $[0,1)\cup (1,\infty)$ and satisfies 
\[
Q(0) = q_0,\;\; Q'(0) = 0
\]
as well a representation \eqref{anear10} (where $|\tilde{q}_2|\ll 1$ depends on $q_0$). We have the asymptotic behavior 
\[
|Q(a)|\lesssim a^{-\frac{1}{3}},\qquad a\rightarrow\infty. 
\]
This function obeys the representation \eqref{anear1} for $a\in [\frac{1}{2},1)$, as well as the representation \eqref{anear10} for $a\in (1, 1+\ell]$ 
with $\ell = \ell(\tilde{q}_1, \tilde{q}_2)$.  
\end{prop}

Proceeding in exact analogy to Section~\ref{sec:surgery1},  we introduce the modified approximate solution 
\EQ{\label{eq:u large}
u(t, r) = t^{-\frac{1}{3}}\chi(t-r)|1-a|^{\frac{2}{3}}X(a) + t^{-\frac{1}{3}}Q_2(1)
}
We attempt to turn this into an actual solution of \eqref{NLW7} (in the defocussing case) by adding a correction term $v(t, r)$. Here $v$ solves \eqref{eq:vequation}. In analogy with Section~\ref{sec:uv} we state two main propositions, the first one being local existence
for the linearized equation about~$u(t,r)$.

\begin{prop}\label{prop:local1}
Let $u$ be as in \eqref{eq:u large} above, and 
assume that $v[T]$ is compactly supported with $$\|v[T]\|_{\dot{H}^{\frac{7}{6}}(\R^{3})\times \dot{H}^{\frac{1}{6}}(\R^3)}\ll 1.$$ 
Then there exists some time $T_1>T$ and a solution 
\[
v\in L_t^\infty\dot{H}^{\frac{7}{6}}([T, T_1]\times \R^3),\;\;v_t\in L_t^\infty \dot{H}^{\frac{1}{6}}([T, T_1]\times \R^3)
\]
of \eqref{eq:vequation} with compact support on every time slice $t\times \R^3$, $t\in [T, T_1]$. If 
$$v[T]\in \dot{H}^s(\R^{3})\times \dot{H}^{s-1}(\R^{3}),$$ then also 
\[
v[t]\in \dot{H}^s(\R^{3})\times \dot{H}^{s-1}(\R^{3})
\]
for all $s>\frac{7}{6}$. 
\end{prop}

This is essentially proved in the same fashion as Proposition~\ref{prop:local},  see the final section.

\begin{prop}
\label{prop:bootstrap1} 
Let  $C\geq 1$ and $\tilde{q}_1$ as in the preceding proposition be fixed, $T\geq 1 $ sufficiently large, and $q_0$ (as in the preceding proposition) sufficiently small. Also, for a sufficiently small $\delta_1$,  suppose that $v[T]$ with support on $r\in [T-C, T+C]$, satisfies 
\[
\|v[T]\|_{\dot{H}^{\frac{7}{6}}\cap\dot{H}^1(\R^{3})\times \dot{H}^{\frac{1}{6}}\cap L^2(\R^{3})} \leq \delta_1\ll1 
\]
Then there exists $C_1>1$ with $C_1\delta_1\ll 1$ as well as $\eps = \eps(q_0)>0$ such that for any $T_1>T$
\[
\|v\|_{L_t^6 L_x^{18}([T, T_1]\times \R^3)} + \sup_{t\in [T, T_1]}\|v[t]\|_{\dot{H}^{\frac{7}{6}}\cap (t-T)^{\eps}\dot{H}^1(\R^{3})\times \dot{H}^{\frac{1}{6}}\cap (t-T)^{\eps}L^2(\R^{3})}\leq C_1\delta_1
\]
implies 
\[
\|v\|_{L_t^6 L_x^{18}([T, T_1]\times \R^3)} + \sup_{t\in [T, T_1]}\|v[t]\|_{\dot{H}^{\frac{7}{6}}\cap (t-T)^{\eps}\dot{H}^1(\R^{3})\times \dot{H}^{\frac{1}{6}}\cap (t-T)^{\eps}L^2(\R^{3})}\leq \frac{C_1}{2}\delta_1.
\]
We may also include any other Strichartz norm on the left-hand side, see Lemma~\ref{lem:ST}. 
\end{prop}
\begin{proof} We proceed in close analogy to the proof of 
Proposition~\ref{prop:bootstrap}, considering first the energy and then the scaling invariant norm.
Notice carefully that the errors $e_j$ will not be small in a pointwise sense.   However, due to the support
condition on~$v$ and by taking the initial time~$T$ sufficiently large, we will see that the influence of 
the errors $e_j$ on the solution can be made as small as we wish. 

\subsection{Energy control}
Observe that for $t\in [T, T_1]$
\begin{equation}\label{eq:enerident1}\begin{split}
&\int_{\R^{3}}\big[\frac{1}{2}\big(v_t^2 + |\nabla v|^2\big) +\frac{7}{2}u^6v^2+\ldots + uv^7 +\frac{1}{8}v^8\big](t, \cdot)\,dx\\
&-\int_{\R^{3}}\big[\frac{1}{2}\big(v_t^2 + |\nabla v|^2\big) +\frac{7}{2}u^6v^2+\ldots+ uv^7 +\frac{1}{8}v^8\big](T, \cdot)\,dx\\
&= \int_T^t\int_{\R^{3}}\big[\sum_{j=1}^3 -e_j v_t +21 u_t u^5 v^2 +\ldots + u_t v^7\big]\,dx dt
\end{split}\end{equation}
Note that 
\[
\int u^6(T, \cdot) v^2(T, \cdot)\,dx\lesssim \|v\|_{\dot{H}^1}^2\int_{T-C}^{T+C}r^{-2}r^{-1}r^2\,dr\lesssim T^{-1}\|v\|_{\dot{H}^1}^2\ll \delta_1^2
\]
provided we choose $T$ large enough (depending on $\tilde{q}_1$ which now influences the size of $u$). 
Moreover, exactly as in subsection~\ref{subsec:ener}, we  obtain for $j \in [1,5]$ the bounds 
\begin{align*}
&\sup_{t\in [T,T_1]}\|u^jv^{8-j}(t, \cdot)\|_{L_x^1}\\&\lesssim T^{\eps\alpha(8-j)-\frac{1}{9}}\big(\sup_{t\in [T,T_1]}(T-t)^{-\eps}\|v(t, \cdot)\|_{\dot{H}^{1}}\big)^{\alpha(8-j)}\|v(t, \cdot)\|_{\dot{H}^{\frac{7}{6}}}^{(1-\alpha)(8-j)}\\
\end{align*}
where here and in the sequel the implicit constant depends on $\tilde{q}_1$, and by choosing $T$ sufficiently large, we can make this 
\[
\ll \delta_1^2. 
\]
 Moreover, the contribution of the pure power term $v^8$ is estimated as before by 
 \begin{align*}
\|v^8(T, \cdot)\|_{L_x^1}&\leq \|v(t, \cdot)\|_{\dot{H}^1}^2\big(\sup_{t\in [1,T]}\|v(t, \cdot)\|_{\dot{H}^{\frac{7}{6}}}^6\big)\\
&\lesssim \|v(t, \cdot)\|_{\dot{H}^1}^2 (C\delta_1)^6\ll \delta_1^2
\end{align*}
It remains to control the terms on the right hand side of \eqref{eq:enerident1}. The contribution of the terms involving the $e_j$ is again straightforward. In fact, just as in subsection~\ref{subsec:ener}, we get 
\[
\big|\int_T^t e_j v_t\,dx dt\big|\lesssim \big(\int_T^t s^{\eps-\frac{4}{3}}\,ds\big)\big(\sup_{t\in [T, T_1]}(T-t)^{-\eps}\|v_t(t, \cdot)\|_{L_x^2}\big)\ll \delta_1^2
\]
provided we choose $T$ sufficiently large. It remains to consider the remaining source terms 
\[
21 u_t u^5 v_t,\ldots, u_t v^7, 
\]
of which only the first one is delicate, as the others all result in gains in $T$, whence the required smallness gain. 
To handle the delicate term, we write (on a fixed time slice)
\[
\int u_t u^5 v^2\,dx = \int_{r<t}u_t u^5 v^2\,dx  + \int_{t<r<t+C} u_t u^5 v^2\,,dx 
\]
Here we have exploited the fact that by the Huyghen's principle, the perturbation $v$ is supported in the neighborhood $r<t+C$ of the forward light cone. 
Since the approximately self-similar $u(t, r)$ is given by the small-data ansatz in the interior of the light cone, we can repeat verbatim the estimates following \eqref{eq:delicate} to conclude that 
\[
\int_{T}^{T_1}\int_{r<t}\big| u_t u^5 v^2\big|\,dxdt\ll \delta_1^2,
\]
provided the $q_0$ is chosen sufficiently small.  Thus consider now the term 
\begin{align*}
\int_{t<r<t+C} u_t u^5 v^2\,dx 
\end{align*}
Using the bound $|u_t|\lesssim t^{-1}$, see \eqref{eq:u_t}, we can bound this by 
\begin{align*}
\big|\int_{t<r<t+C} u_t u^5 v^2\,dx \big|&\lesssim \|v\|_{\dot{H}^1}^2 t^{-\frac{8}{3}}\int_{t<r<t+C}r^{-1}r^2\,dr\\&
\lesssim 
t^{2\eps - \frac{5}{3}}\big(\sup_{t\in [T, T_1]}(T-t)^{-\eps}\|v(t, \cdot)\|_{\dot{H}^1}\big)^2,
\end{align*}
and so we infer 
\begin{align*}
\int_{T}^{T_1}\int_{t<r<t+C}\big| u_t u^5 v^2\big|\,dxdt&\lesssim \int_{T}^{T_1}t^{2\eps - \frac{5}{3}}\,dt\big(\sup_{t\in [T, T_1]}(T-t)^{-\eps}\|v(t, \cdot)\|_{\dot{H}^1}\big)^2\\
&\ll \delta_1^2,
\end{align*}
provided  $T$ is sufficiently large. 

The expression 
\[
\int_{T}^t\int u_t v^7\,dxdt,
\]
as well as all the ``intermediate source terms'' in \eqref{eq:enerident1}, are again all estimated just as in subsection~\ref{subsec:ener}, resulting in gains  of  a factor  $T^{-1}$ which furnishes the required smallness. 
This completes the bootstrap for the energy norm 
\[
\sup_{t\in [T, T_1]}(T-t)^{-\eps}\|v(t, \cdot)\|_{\dot{H}^1} + \sup_{t\in [T, T_1]}(T-t)^{-\eps}\|v_t(t, \cdot)\|_{L^2}.
\]
\subsection{Critical norm control} 
We  repeat the estimates from subsection~\ref{subsec:critnorm}, which are all seen to result in a gain of a factor~$T^{-1}$ (for the nonlinear source terms), and so the bootstrap is immediate by choosing $T$ large enough. 
This completes the proof of Proposition~\ref{prop:bootstrap1}.
\end{proof}

\subsection{Proofs of Theorems~\ref{thm:Main2}, \ref{thm:intro2}}
Invoking Propositions~\ref{prop:local1}, \ref{prop:bootstrap1} we have shown that the approximate solution $u(t, r)$ can be completed to an exact global-in-forward time solution 
\[
\tilde{u}(t, r) = u(t, r) + v(t,r). 
\]
Moreover, this solution preserves any additional regularity of the data $v[T]$ above $\dot{H}^{\frac{7}{6}}(\R^3)\times \dot{H}^{\frac{1}{6}}(\R^3)$. 

Translating the time $t = T$ to time $t = 0$, picking $\tilde{q}_1\gg  M^9$ (see Lemma~\ref{lem:near10})  and re-scaling 
\[
\tilde{u}(t, r)\longrightarrow T^{\frac{1}{3}}\tilde{u}(T t, Tr),
\]
we have now shown Theorem~\ref{thm:Main2}.  The largeness condition~\eqref{fM} is an immediate consequence of the 
estimate \eqref{Q gross} and the fact that we may choose the initial data so as not to destroy this pointwise property. 

Theorem~\ref{thm:intro2} is proved by truncation, analogously to the way in which we obtained Theorem~\ref{thm:intro1}. 
Once again, finite time blowup can only occur at the origin due to the pointwise a priori bound for all $r>0$ (fixed) uniformly in time
as a result of the conserved positive definite energy. 

\section{Local solvability of the perturbative equation}\label{sec:local}

Here we prove Proposition~\ref{prop:local}. Thus let $v[T]$ be as in that proposition. We immediately observe from their definition that the errors $e_j$ have compact support on fixed time slices, and hence the compact (spatial) support of $v[T]$ will imply that of $v[t]$ for any $t$. We construct $v$ as a limit of the iterative process 
\[
-v^{(j)}_{tt} + \triangle v^{(j)} \mp 7u^6v^{(j-1)} \mp \ldots \mp 7u(v^{(j-1)})^6 \mp (v^{(j-1)})^7 = \sum_{i=1}^3 e_i,\;\; j\geq 1,\,
\]
\[
v^{(0)}(\cdot) = S(\cdot)(v[T]),
\]
where $S(\cdot)$ denotes the standard free wave propagator. We shall assume that 
$$\|v[T]\|_{\dot{H}^{\frac{7}{6}}\times \dot{H}^{\frac{1}{6}}(\R^3)} \le \delta$$ where $\delta$ 
is some small but absolute constant, and then show that the sequence $v^{(j)}$ 
converges in $$L_t^6 L_x^{18}(\R^3)\cap L_t^\infty\dot{H}^{\frac{7}{6}}(\R^3)$$ on the time slice $[T, T+1]\times \R^3$. We may assume that
\[
\|u\|_{L_t^6 L_x^{18}([T, T+1]\times \R^3)}\leq C_1\delta,\;\;\;\sum_{i=1}^3\|e_i\|_{L_t^1 \dot{H}^{\frac{1}{6}}([T, T+1]\times\R^3)}\leq C_2\delta
\]
for some constants $C_{1,2}$, uniformly in $T\geq 1$ (see Corollary~\ref{cor:3}). 
We conclude from Strichartz' inequality, see Lemma~\ref{lem:ST},  that 
\begin{align*}
&\|v^{(j)}\|_{L_t^6 L_x^{18}\cap L_t^\infty\dot{H}^{\frac{7}{6}}(\R^3)} +  \|\partial_t v^{(j)}\|_{L_t^\infty\dot{H}^{\frac{1}{6}}(\R^3)}
\\&\leq C_3\big[\sum_{k=0}^{6}\|\,|\nabla|^{\frac{1}{6}}(u^k (v^{(j-1)})^{7-k})\|_{L_t^1 L_x^2([T, T+1]\times\R^3)} + \sum_{i=1}^3\|e_i\|_{L_t^1 \dot{H}^{\frac{1}{6}}(\R^3)}\big]
\end{align*}
We have 
\EQ{\label{eq:vj7}
\| (v^{(j-1)})^7\|_{L_t^1 \dot{H}^{\frac{1}{6}}([T, T+1]\times \R^3)}\leq C_4 \|v^{(j-1)}\|_{(L_t^6 L_x^{18}\cap L_t^\infty\dot{H}^{\frac{7}{6}})([T, T+1]\times \R^3)}^7
}
Indeed, by the fractional Leibnitz rule,
\EQ{\nn
\|f^7\|_{\dot H^{\f16}(\R^3)} &\less \|f\|_{\dot W^{\f16,6}(\R^3)} \|f^6\|_{L^3(\R^3)} \less \|f\|_{\dot{H}^{\frac{7}{6}}(\R^3)} \|f\|_{L^{18}(\R^3)}^6
}
Integrating this in time over $[T,T+1]$ yields \eqref{eq:vj7}. 
By the same type of reasoning, if $k = 1,2,\ldots, 6$, then we have 
\begin{align*}
&\|u^k (v^{(j-1)})^{7-k}\|_{L_t^1 \dot{H}^{\frac{1}{6}}([T, T+1]\times \R^3)}\\
&\leq C_5\|\,|\nabla|^{\frac{1}{6}}u\|_{L_t^6 L_x^9([T, T+1]\times \R^3)}\|u\|_{L_t^6 L_x^{18}([T, T+1]\times \R^3)}^{k-1}\times\\
&\qquad \qquad \times\|v^{(j-1)}\|_{L_t^6 L_x^{18}([T, T+1]\times \R^3)}^{6-k}\|v^{(j-1)}\|_{L_t^\infty \dot{H}^{\frac{7}{6}}([T, T+1]\times \R^3)}\\
& + C_6\|u\|_{L_t^6 L_x^{18}([T, T+1]\times \R^3)}^{k}\|v^{(j-1)}\|_{L_t^6 L_x^{18}([T, T+1]\times \R^3)}^{6-k}\|\,|\nabla|^{\frac{1}{6}}v^{(j-1)}\|_{L_t^\infty \dot{H}^1([T, T+1]\times \R^3)}
\end{align*}
where we have also used the Sobolev embedding (in the context of functions vanishing at infinity on $\R^3$)
\[
\dot{H}^1(\R^3)\subset L^6(\R^3)
\]
Note that $\,|\nabla|^{\frac{1}{6}}u(t, \cdot)\in L_x^9$ due to symbolic behavior with respect to $r$ for $r\gg t$. 
It then follows that provided we make the inductive assumption 
\[
\|v^{(j-1)}\|_{L_t^6 L_x^{18}\cap L_t^\infty\dot{H}^{\frac{7}{6}}([T, T+1]\times \R^3)} \leq K\delta
\]
for some sufficiently large constant $K$ (independent of $\delta$), we obtain that 
\[
\|v^{(j)}\|_{L_t^6 L_x^{18}\cap L_t^\infty\dot{H}^{\frac{7}{6}}([T, T+1]\times \R^3)}\leq C_7\,K^7\delta^7+C_8\,\delta
\]
where we have exploited that  
\[
\sum_{i=1}^3\|e_i\|_{L_t^1 \dot{H}^{\frac{1}{6}}}\leq C_8\,\delta
\]
as well as 
\[
\|u\|_{L_t^6 L_x^{18}\cap L_t^6\,|\nabla|^{-\frac{1}{6}}L_x^9([T, T+1]\times \R^3)}\leq C_9\,\delta
\]
from our choice of $\delta$. 
Choosing $\delta>0$ small enough in relation to $C_7$ and $K$ large enough in relation to $C_8$, we obtain 
\[
\|v^{(j)}\|_{L_t^6 L_x^{18}\cap L_t^\infty\dot{H}^{\frac{7}{6}}([T, T+1]\times \R^3)}\leq K\delta
\]
and thus we get the desired a priori bound. Passing to the difference equation yields the convergence of the $v^{(j)}$. 
The higher derivative bounds follow in standard fashion by differentiating the equation for $v^{(j)}$. 
This completes the proof of Proposition~\ref{prop:local}. 

\bigskip

Next, in order to prove Proposition~\ref{prop:local1}, the main difference lies with the fact that the function $u$ is no longer small. 
Thus in order to ensure convergence of the iteration, one needs to replace the interval $[T, T+1]$ by one of the form $[T, T+\kappa]$ where $\kappa = \kappa(u)$ depends on 
$$\sup_{t}\|\,|\nabla|^{\frac{1}{6}}u\|_{L_x^9} + \sup_{t}\|u\|_{L_x^{18}}$$
 Otherwise, the argument is identical to the preceding one.

\bigskip

\centerline{\scshape Joachim Krieger }
\medskip
{\footnotesize
 \centerline{B\^{a}timent des Math\'ematiques, EPFL}
\centerline{Station 8, 
CH-1015 Lausanne, 
  Switzerland}
  \centerline{\email{joachim.krieger@epfl.ch}}
} 

\medskip

\centerline{\scshape Wilhelm Schlag}
\medskip
{\footnotesize
 \centerline{Department of Mathematics, The University of Chicago}
\centerline{5734 South University Avenue, Chicago, IL 60615, U.S.A.}
\centerline{\email{schlag@math.uchicago.edu}
}
} 

\end{document}